\documentclass{amsart}
\usepackage{marktext}
\usepackage{gimac}
\RequirePackage[nointlimits]{esint}

\newcommand{\GI}[2][]{\sidenote[colback=yellow!20]{\textbf{GI\xspace #1:} #2}}

\newcommand{\eff}{\mathrm{eff}}
\newcommand{\tmix}{t_\mathrm{mix}}

\newcommand{\e}{\bm{e}}

\newcommand{\cyl}{\mathcal{C}}
\newcommand{\zerotuple}{{\bm{0}}}

\DeclareMathOperator{\dist}{dist}

\DeclareMathOperator{\cov}{cov}
\DeclareMathOperator{\unif}{unif}

\newtheorem{assumption}[theorem]{Assumption}

\begin{document}
\title[Residual diffusivity]{Residual Diffusivity for Noisy Bernoulli Maps}
\begin{abstract}
  Consider a discrete time Markov process~$X^\epsilon$ on~$\R^d$ that makes a deterministic jump prescribed by a map~$\varphi \colon \R^d \to \R^d$, and then takes a small Gaussian step of variance~$\epsilon^2$.
  For certain chaotic maps~$\varphi$, the \emph{effective diffusivity} of~$X^\epsilon$ may be bounded away from~$0$ as~$\epsilon \to 0$.
  This is known as \emph{residual diffusivity},
  and in this paper we prove residual diffusivity occurs for a class of maps~$\varphi$ obtained from piecewise affine expanding Bernoulli maps.
\end{abstract}
\thanks{This work has been partially supported by the National Science Foundation under grants
  DMS-2108080 DMS-2406853,
  and the Center for Nonlinear Analysis.}
\author[Iyer]{Gautam Iyer}
\address{%
  Department of Mathematical Sciences, Carnegie Mellon University, Pittsburgh, PA 15213.
}
\email{gautam@math.cmu.edu}
\author[Nolen]{James Nolen}
\address{%
  Department of Mathematics, Duke University, Durham, NC 27708.
}
\email{nolen@math.duke.edu}
\maketitle
\section{Introduction}\label{s:intro}

\subsection{Main Result}

The goal of this paper is to study the long time behavior of the Markov process $\{X^\epsilon_n\}_{n\geq 0}$ on $\R^d$ defined by
\begin{equation}\label{e:Xn}
  X^\epsilon_{n+1} = \varphi(X^\epsilon_n) + \epsilon \xi_{n+1} \,.
\end{equation}
Here $\{\xi_n\}_{n\geq 1}$ is a family of independent standard Gaussian random variables, and~$\varphi \colon \R^d \to \R^d$ is a Lebesgue measure preserving map with a periodic displacement (i.e.\ the function~$x \mapsto \varphi(x) - x$ is~$\Z^d$ periodic).
Since~$\varphi$ has a periodic displacement, it projects to a well-defined map~$\tilde \varphi \colon \T^d \to \T^d$ defined by
\begin{equation}
  \tilde \varphi(x) = \varphi(x) \pmod{\Z^d}
  \,.
\end{equation}

\emph{Residual diffusivity} is the phenomenon that the asymptotic variance (i.e.\ the limit as~$n \to \infty$ of~$\var(\e_i \cdot X_n) / n$) remains bounded and non-degenerate as~$\epsilon \to 0$.
It has been conjectured~\cite{Taylor21,BiferaleCrisantiEA95,MurphyCherkaevEA17} to occur in certain situations
when~$\tilde \varphi \colon \T^d \to \T^d$ is ``chaotic''.
The main result of this paper shows that the processes~$X^\epsilon$ exhibit residual diffusion when~$\varphi$ is obtained from a piecewise affine linear, expanding Bernoulli map (see Section~\ref{s:bernoulli}, below).
To the best of our knowledge, this is the only chaotic map for which residual diffusion has been rigorously proved.

\begin{theorem}\label{t:residual-diffusivity}
  Suppose~$\varphi$ is obtained from an expanding, piecewise affine linear, Bernoulli map satisfying the conditions in Assumption~\ref{a:phi}, below.
  There exists a constant $c > 0$ such that for all $v \in \R^d$, and all initial distributions~$\mu$ we have
  \begin{equation}\label{e:residual-diffusivity}
    \liminf_{\epsilon \to 0}
      \lim_{n \to \infty} \frac{\var^\mu( v \cdot X^\epsilon_n )}{n}
      \geq c \var\paren[\big]{ v \cdot \floor{\varphi(U)} \given U \sim \unif( [0, 1)^d ) } 
      \,.
  \end{equation}
\end{theorem}

We use $\var^\mu(v \cdot X^\epsilon_n)$ to denote the variance of~$v \cdot X^\epsilon_n$ given~$X^\epsilon_0 \sim \mu$.
For $x \in \R^d$ the notation $\floor{x}$ used above denotes the unique $n \in \Z^d$ such that $x \in n + [0, 1)^d$.
Unless $\floor{\varphi(x)}$ is a constant for (almost every) $x \in [0,1)^d$, the right side of \eqref{e:residual-diffusivity} will be positive.  A heuristic argument, supported by numerics, indicates that the bound~\eqref{e:residual-diffusivity} is sharp.
While we are currently unable to prove a matching upper bound, we can prove an upper bound that grows like~$O(\abs{\ln \epsilon})$ as~$\epsilon \to 0$.

\begin{proposition}\label{p:var-upper-intro}
  Let~$\varphi$ be the same map from Theorem~\ref{t:residual-diffusivity}.
  There exists a constant~$C < \infty$ such that for all~$v \in \R^d$ and all sufficiently small~$\epsilon > 0$ we have
  \begin{equation}\label{e:residual-diffusivity-upper}
      \lim_{n \to \infty} \frac{\var( v \cdot X^\epsilon_n )}{n}
      \leq C \abs{\ln \epsilon} \abs{v}^2
      \,.
  \end{equation}
\end{proposition}

\subsection{Motivation and Literature review}

Our interest in this problem stems from understanding the long time behavior of diffusions with a chaotic (but periodic) drift.
That is, consider the continuous time diffusion process $X^\epsilon_t$ defined by the SDE
\begin{equation}\label{e:SDEX}
  d X^\epsilon_t = u(X^\epsilon_t) \, dt + \epsilon \, dW_t
  \quad\text{on } \R^d
  \,,
\end{equation}
where $u \colon \R^d \to \R^d$ is a periodic and divergence-free vector field.
One physical situation where this is relevant is in the study of diffusive tracer particles being advected by an incompressible fluid.

On small (i.e.\ $O(1)$) time scales, it the process of~$X^\epsilon$ stays close to the deterministic trajectories of~$u$, and a large deviations principle can be established (see for instance~\cite{FreidlinWentzell12}).
On intermediate (i.e.\ $O(\abs{\ln \epsilon}/\epsilon^\alpha)$ for~$\alpha \in [0, 2)$) time scales certain non-Markovian effects arise and lead to anomalous diffusion~\cite{
  Young88,
  YoungPumirEA89,
  Bakhtin11,
  HairerKoralovEA16,
  HairerIyerEA18
}.
On long (i.e.\ $O(1/\epsilon^2)$) time scales, homogenization occurs and the net effect of the drift can be averaged and the process~$X^\epsilon$ can be approximated by a Brownian motion with covariance matrix~${D^\epsilon_\eff}$ called the \emph{effective diffusivity}.
This was first studied in this setting by Freidlin~\cite{Freidlin64}, and is now the subject of many standard books with several important applications~\cite{BensoussanLionsEA78,PavliotisStuart08}.

The effective diffusivity matrix ${D^\epsilon_\eff}$ is formally given by
\begin{equation}
  \e_i {D^\epsilon_\eff} \e_j
     = \lim_{t \to \infty} \frac{ \cov( \e_i \cdot X^\epsilon_t, \e_j \cdot X^\epsilon_t ) }{t}
     \,,
\end{equation}
where~$\e_i \in \R^d$ is the $i^\text{th}$ standard basis vector.
This, however, is hard to compute explicitly and authors usually characterize it in terms of a cell problem on $\T^d$.
In a few special situations (such as shear flows, or cellular flows) the asymptotic behavior of~${D^\epsilon_\eff}$ as~$\epsilon \to 0$ is known~\cite{
  Taylor53,
  ChildressSoward89,
  FannjiangPapanicolaou94,
FannjiangPapanicolaou97,
MajdaKramer99,
  Heinze03,
  Koralov04,
NovikovPapanicolaouEA05,
  RyzhikZlatos07
}.

The motivation for the present paper comes from thinking about the case when the deterministic flow of~$u$ is chaotic or exponentially mixing on the torus (see~\cite{SturmanOttinoEA06}).
In this case it has been conjectured that~${D^\epsilon_\eff}$ is~$O(1)$ as~$\epsilon \to 0$, a phenomenon known as ``residual diffusivity''.
Study of this was initiated by Taylor~\cite{Taylor21} over 100 years ago, and has since been extensively studied by many authors~\cite{
  ZaslavskyStevensEA93,
  BiferaleCrisantiEA95,
  MajdaKramer99,
  Zaslavsky02,
  MurphyCherkaevEA17
}.
While this has been confirmed numerically and studied for elephant random walks~\cite{
  LyuXinEA17,
  LyuXinEA18,
  MurphyCherkaevEA20,
  WangXinEA21,
  WangXinEA22,
  LyuWangEA22,
  KaoLiuEA22
},
a rigorous proof is of this in even one example is still open.

We now present a heuristic explanation as to why residual diffusivity is expected.
As before, let~$\tilde X^\epsilon$ be the projection of~$X^\epsilon$ to the torus~$\T^d$, and let~$T_\epsilon = \tmix(\tilde X^\epsilon)$ be the \emph{mixing time} of~$\tilde X^\epsilon$ on the torus~$\T^d$.
In general, due to the Gaussian noise, we are guaranteed~$T_\epsilon \leq O( 1/\epsilon^2)$.
The chaotic dynamics of~$\varphi$ may cause \emph{enhanced dissipation}, and reduce~$T_\epsilon$ significantly (see for instance~\cite{
  FannjiangWoowski03,
  FannjiangNonnenmacherEA04,
  FannjiangNonnenmacherEA06,
  ConstantinKiselevEA08,
  FengIyer19,
  CotiZelatiDelgadinoEA20,
  ElgindiLissEA23,
  IyerLuEA24,
  TaoZworski24
}).
However, for reasons that will be explained shortly, having~$T_\epsilon \ll 1/\epsilon^2$ does not necessarily imply residual diffusion.

Now, after the mixing time~$T_\epsilon$ the distribution of~$\tilde X^\epsilon$ is close to the uniform distribution on~$\T^d$.
Thus, it is reasonable to expect that the distribution of~$X^\epsilon$ can be approximated by a linear combination of indicator functions of unit cubes in~$\R^d$ whose vertices lie in the integer lattice~$\Z^d$ (see Figure~\ref{f:Nepsilon}, right).
By the Markov property this should imply that for every~$v \in \R^d$ we have
\begin{equation}\label{e:varXTep}
  \lim_{t \to \infty} \frac{\var(v \cdot X_t^\epsilon)}{t}
    = \frac{ \var(v \cdot X^\epsilon_{T_\epsilon} \given X^\epsilon_0 \sim \unif( [0, 1)^d ) }{T_\epsilon}
    \,.
\end{equation}

Let~$X^0$ be the (deterministic) solution to~\eqref{e:SDEX} with~$\epsilon = 0$.
If we knew
\begin{equation}\label{e:XepMinusX0}
  \abs{X^\epsilon_{T_\epsilon} - X^0_{T_\epsilon}} = o(T_\epsilon)
\end{equation}
then
\begin{equation}
  \frac{ \var(v \cdot X^\epsilon_{T_\epsilon} \given X^\epsilon_0 \sim \unif( [0, 1)^d ) }{T_\epsilon}
  \approx
  \frac{ \var(v \cdot X^0_{T_\epsilon} \given X^0_0 \sim \unif( [0, 1)^d ) }{T_\epsilon}
  \,.
\end{equation}
If the deterministic flow of~$u$ is chaotic and if~$X^0_0 \sim \unif( [0, 1)^d )$ then~$X^0_n$ should behave like a random walk after~$n$ steps.
Since the variance of a random walk grows linearly with time, we expect
\begin{equation}\label{e:rdiff-cts}
  \lim_{\epsilon \to 0} \lim_{t \to \infty} \frac{\var(v \cdot X_t^\epsilon)}{t}
    = \var(v \cdot X^0_{1} \given X^0_0 \sim \unif( [0, 1)^d )
    \,.
\end{equation}

Unfortunately the above heuristic has not been proved for even one example, and there are two main technical obstructions to making it rigorous.
First, after time~$T_\epsilon$, even though the density of~$\tilde X^\epsilon_{T_\epsilon}$ is roughly uniform on~$\T^d$, the density of~$X^\epsilon_{T_\epsilon}$ need not be well approximated by a linear combination of indicator functions of unit cubes with vertices on the integer lattice.
Second, the distance between~$X^\epsilon_t$ and~$X^0_t$ may grow exponentially with~$t$, and the bound~\eqref{e:XepMinusX0} need not hold.
\medskip

The goal of this paper is to rigorously exhibit residual diffusivity in a simple setting.
In continuous time, examples of exponentially mixing flows are not easy to construct.
The canonical example of an exponentially mixing flows is the geodesic flow on the unit sphere bundle of negatively curved manifolds~\cite{Dolgopyat98}. On the~$3$-torus, however, the existence of a divergence free, $C^1$, time independent, exponentially mixing velocity field is an open question.
To the best of our knowledge, there are only examples of lower regularity~\cite{ElgindiZlatos19}, and several time dependent examples~\cite{Pierrehumbert94,BedrossianBlumenthalEA19,MyersHillSturmanEA22,BlumenthalCotiZelatiEA23,ElgindiLissEA23,ChristieFengEA23}.

On the other hand, there are several simple, well known, examples of regular, Lebesgue measure preserving, exponentially mixing dynamical systems on the torus~\cite{SturmanOttinoEA06}.  Therefore, instead of studying a continuous-time system, we study the discrete time system~\eqref{e:Xn}, and choose~$\varphi$ so that its projection to the torus generates an exponentially mixing dynamical system.
Such systems are interesting in their own right, and various aspects of them have been extensively studied~\cite{
  FannjiangWoowski03,
  FannjiangNonnenmacherEA04,
  ThiffeaultChildress03,
  FengIyer19,
  OakleyThiffeaultEA21,
  IyerLuEA24
}.

In this time-discrete setting, the analog of~\eqref{e:rdiff-cts} is
\begin{equation}\label{e:RDiffChaotic}
  \lim_{\epsilon \to 0}
    \lim_{n \to \infty} \frac{\var^x( v \cdot X^\epsilon_n )}{n}
      = \var\paren[\big]{ v \cdot \floor{\varphi(U)} \given U \sim \unif( [0, 1)^d ) }
    \,.
\end{equation}
Our main result (Theorem~\ref{t:residual-diffusivity}) establishes a lower bound that obtains~\eqref{e:RDiffChaotic} to up to a constant factor.
For the upper bound we will prove the discrete time version of~\eqref{e:varXTep}.
However, the best estimate we can presently obtain on the right hand side of~\eqref{e:varXTep} is suboptimal, leading to the logarithmically growing upper bound in Proposition~\ref{p:var-upper-intro}.

\subsection*{Plan of this paper}
In Section~\ref{s:bernoulli} we describe piecewise linear expanding Bernoulli maps, and precisely state the assumptions required for Theorem~\ref{t:residual-diffusivity}.
Next, in Section~\ref{s:variance-upper} we prove a general upper bound for the asymptotic variance of Markov processes with a transition density that is invariant under~$\Z^d$ shifts.
We use this bound to prove the upper bound Proposition~\ref{p:var-upper-intro}.
The proof of the lower bound can be broken down into two steps.
The first step is to prove a lower bound for the asymptotic variance of Markov processes whose transition density can be suitably minorized.
The second step is to prove that the process we study satisfies this minorization condition.
Section~\ref{s:lower} states these two steps precisely, and proves Theorem~\ref{t:residual-diffusivity} assuming these steps.
Section~\ref{s:KV} uses an argument of Kipnis and Varadhan to prove the first step, and finally Section~\ref{s:minorizing} proves the required minorization condition for the processes~$X^\epsilon$.

\subsection*{Acknowledgements}

The authors would like to thank Jack Xin and Albert Fannjiang for helpful discussions.

\section{Piecewise Affine Linear Expanding Bernoulli Maps.}\label{s:bernoulli}

We begin by precisely describing the piecewise affine linear, expanding Bernoulli map~$\varphi:\R^d \to \R^d$, and stating the assumptions required for Theorem~\ref{t:residual-diffusivity}.
Partition $\R^d$ into unit cubes $\set{Q_k \st k \in \Z^d}$, where $Q_k = k + [0, 1)^d$.
Let $M \geq 2$ and $E_1, \dots, E_M \subseteq Q_0$ be a partition of $Q_0$, with
\begin{equation}
  \abs{E_1} \leq \abs{E_2} \cdots \leq \abs{E_M}
  \,.
\end{equation}
For each $i \in \set{1, \dots, M}$ let~$\varphi_i:E_i \to \R^d$ be an affine linear map of the form
\begin{equation}\label{e:Phi-i-Def}
  \varphi_i(x) = \frac{\mathcal O_i x}{\abs{E_i}^{1/d}} + o_i\,, \quad \forall \; x \in E_i,
\end{equation}
for some vectors~$o_i \in \R^d$, and orthogonal matrices~$\mathcal O_i$ (one example of such a map is shown in Figure~\ref{f:phi}).
Given $x \in \R^d$ we let $n = \floor{x}$ denote the unique element in $\Z^d$ such that $x \in Q_n = n + Q_0$, and define
\begin{equation}\label{e:PhiDef}
  \varphi(x) = n + \varphi_i(x - n) 
  \quad \text{if } x - n \in E_i
  \,.
\end{equation}
The function $x \mapsto \varphi(x) - x$ is $\Z^d$ periodic, and~$\varphi$ is Lebesgue measure preserving.
We now precisely describe the assumptions required for Theorem~\ref{t:residual-diffusivity}.
\begin{assumption}\label{a:phi}
  Let~$\varphi\colon \R^d \to \R^d$ be the affine linear expanding map defined by~\eqref{e:PhiDef}, and assume that the following conditions hold.
  \begin{enumerate}[(1)]
    \item 
      If $d = 1$, then each $E_i$ is an interval.
      If $d > 1$, then each $E_i$ is a non-degenerate $d$-dimensional cube with edges parallel to the coordinate axes (as in the left figure in Figure~\ref{f:phi}).
    \item\label{a:int-cube}
      For each~$i \in \set{1, \dots, M}$ there exists~$\sigma_0(i) \in \Z^d$ such that $\varphi \colon \mathring E_i \to \mathring Q_{\sigma_0(i)}$ is a bijection. ($\mathring E$ denotes the interior of a set $E$.) 
    \item
      \label{a:boundary}
      For every $x \in \partial Q_0$, there exists $k \in \Z^d$, $y \in \mathring Q_k$ such that $\varphi(y) = x$.
  \end{enumerate}
\end{assumption}

\begin{remark}
  The right hand side of~\eqref{e:residual-diffusivity} is strictly positive if there exists some $i, j \in \set{1, \dots, M}$ such that $\sigma_0(i) \neq \sigma_0(j)$.  This is equivalent to saying that the image of $Q_0$ under $\varphi$ does not coincide with $Q_k$ for some $k \in \Z^d$. 
\end{remark}

\begin{figure}[hbt]
  \includegraphics[width=.4\linewidth]{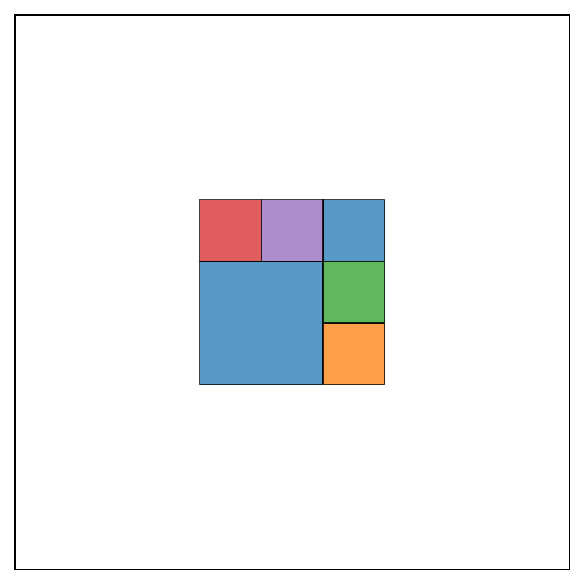}
  \qquad
  \includegraphics[width=.4\linewidth]{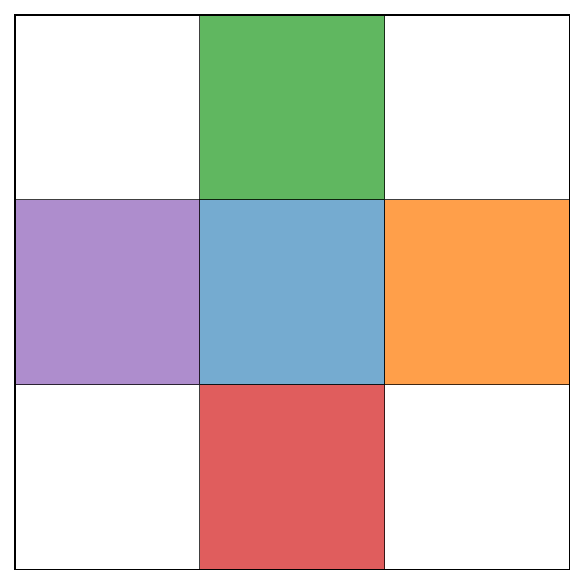}
  \caption{A visual example of~$\varphi$ in two dimensions.
    Each square region in the left figure is expanded to the corresponding unit cube in the right figure.
  }
  \label{f:phi}
\end{figure}

\section{Proof of the upper bound (Proposition \ref{p:var-upper-intro})}\label{s:variance-upper}

The upper bound~\eqref{e:residual-diffusivity-upper} follows from a more general fact about Markov processes whose transition density is invariant under~$\Z^d$ shifts.
Since the proof is short and elementary we present it first.

Let $Y_n$ be a Markov process on $\R^d$ and~$p^Y_n$ denote its~$n$-step transition density.
For brevity, when~$n = 1$ we will write~$p^Y$ for~$p^Y_1$.
We will assume~$p^Y$ is \emph{invariant under~$\Z^d$-shifts}.
That is, for every $k \in \Z^d$, $y, z \in \R^d$ we assume
\begin{equation}\label{e:shift-invariant}
  p^Y(y + k, z + k) = p^Y(y, z)\,.
\end{equation}

Let $\tilde Y_n$, defined by
\begin{equation}\label{e:tildeY}
  \tilde Y_n \defeq Y_n \pmod{\Z^d}
  \,,
\end{equation}
denote the projection of the process~$Y$ onto the torus~$\T^d$.
The $\Z^d$-shift invariance~\eqref{e:shift-invariant} implies that the process~$\tilde Y$ is also a Markov process on~$\T^d$ with transition density given by
\begin{equation}\label{e:ptildeY}
  p^{\tilde Y}(\tilde y, \tilde z)
    = \sum_{k \in \Z^d} p^{Y}(y, z + k)
\end{equation}
where~$y, z \in \R^d$ are any two elements such that
\begin{equation}
  \tilde y = y \pmod{\Z^d}
  \,,
  \qquad\text{and}\qquad
  \tilde z = z \pmod{\Z^d}
  \,.
\end{equation}
We will assume that the stationary distribution of~$\tilde Y$ is the Lebesgue measure on~$\T^d$.
By~\eqref{e:ptildeY} this is equivalent to assuming
\begin{equation}\label{e:leb-inv}\noeqref{e:leb-inv}
    \sum_{n \in \Z^d}  \int_{y \in Q_0} p^Y(y, z + n) \, dy
    = 1
    \quad
    \text{for every~$z \in Q_0$}
  \,.
\end{equation}
By~\eqref{e:shift-invariant}, is equivalent to the condition
\begin{equation}\label{e:leb-inv-prime}
  \tag{\ref{e:leb-inv}$'$}
  \int_{\R^d} p^Y(y, z) \, dy
    = 1
    \,,
    \quad
    \text{for every~$z \in \R^d$}
  \,.
\end{equation}

We will now prove Proposition~\ref{p:var-upper-intro} by proving an upper bound on the asymptotic variance of~$Y$ in terms of the \emph{mixing time} of~$\tilde Y$.
\begin{proposition}\label{p:var-gen-upper}
  Let~$T = \tmix(\tilde Y)$ be the first time for which
  \begin{equation}
    \sup_{\tilde y \in \T^d}
      \norm{p_T^{\tilde Y}(\tilde y, \cdot) - 1 }_{L^1(\T^d)} < \frac{1}{2}
    \,.
  \end{equation}
  There exists a universal constant~$C$ such that for all $v \in \R^d$
  \begin{equation}\label{e:var-gen-upper}
    \lim_{n \to \infty} \frac{ \var^x(v \cdot Y_n) }{n}
      \leq C \abs{v}^2 \tmix(\tilde Y) \sup_{y \in Q_0} \E^y \abs{Y_1 - Y_0}^2\,.
  \end{equation}
\end{proposition}
\begin{remark}
  In this generality the upper bound in Proposition~\ref{p:var-gen-upper} can not be improved and it is easy to produce examples where the upper bound~\eqref{e:var-gen-upper} is attained.
  One such example can be obtained using the flow map of a \emph{shear flow}.
  Explicitly, let~$f \colon \R \to \R$ be a periodic function.
  The flow map at time $t=1$ of the shear flow with profile~$f$ directed along the~$x$ axis is given by
  \begin{equation}
    \varphi(x, y) = (x + f(y), y), \quad \quad (x,y) \in \R^2
    \,.
  \end{equation}
  If we choose~$Y = X^\epsilon$ (i.e. the system~\eqref{e:Xn} with $\varphi$ defined in this way), then~$\tmix(\tilde Y) = O(1/\epsilon^2)$, and if~$v = \e_1$, then both sides of~\eqref{e:var-gen-upper} are~$O(1/\epsilon^2)$ as~$\epsilon \to 0$.
\end{remark}

Proposition~\ref{p:var-gen-upper} immediately implies Proposition~\ref{p:var-upper-intro}, and we now present the proof.
Here, and throughout this paper, we will allow~$C$ to be a generic finite~$\epsilon$-independent constant, whose value is unimportant and may increase from line to line.
\begin{proof}[Proof of Proposition~\ref{p:var-upper-intro}]
  Let~$Y_n = X^\epsilon_n$.
  The upper bound~\eqref{e:var-gen-upper} implies that for any~$v \in \R^d$ we have
  \begin{equation}
    \lim_{n \to \infty} \frac{ \var^x( v \cdot X^\epsilon_n ) }{n}
      \leq C \abs{v}^2 \tmix(\tilde X^\epsilon) \sup_{y \in Q_0} \E \abs{\varphi(y + \epsilon \xi_1) - y}^2
      \,.
  \end{equation}
  By~\cite{IyerLuEA24} we know~$\tmix(\tilde X^\epsilon) \leq C \abs{\ln \epsilon}$, from which~\eqref{e:residual-diffusivity-upper} follows immediately.
\end{proof}

\begin{remark}
  Before proceeding further, we briefly comment on why we expect the upper bound~\eqref{e:residual-diffusivity-upper} is not optimal in our setting and misses by a~$\abs{\ln \epsilon}$ factor.
  Let~$T' = \tmix( \tilde X^\epsilon )$ be the mixing time of~$\tilde X^\epsilon$ on the torus~$\T^d$.
  By~\cite{IyerLuEA24} we recall that~$T' \leq C \abs{\ln \epsilon}$.
  Applying Proposition~\ref{p:var-gen-upper} to the Markov processes~$Y_n = X^\epsilon_{n T'}$ gives
  \begin{equation}
    \lim_{n \to \infty} \frac{ \var^x( v \cdot X^\epsilon_n ) }{n}
      = \frac{1}{T'}
	\lim_{n \to \infty} \frac{ \var^x( v \cdot Y_n ) }{n}
      \leq C \abs{v}^2 \frac{\tmix(\tilde Y)}{T'} \sup_{y \in Q_0} \E^y \abs{X^\epsilon_{T'} - y}^2
      \,.
  \end{equation}
  By definition of~$T'$, we know~$\tmix(\tilde Y) = 1$.
  If~$\varphi$ is sufficiently mixing, then successive steps of~$X^\epsilon$ should decorrelate. 
  In this case, in the absence of a drift, we expect~$\E^y \abs{X^\epsilon_{T'} - y}^2 \leq C T'$,
  from which we immediately obtain
  \begin{equation}
    \lim_{n \to \infty} \frac{ \var^x( v \cdot X^\epsilon_n ) }{n}
      \leq C \abs{v}^2\,.
  \end{equation}
  While this suggests an~$\epsilon$-independent upper bound for the asymptotic variance, it seems challenging to prove rigorously.
\end{remark}

We will now prove Proposition~\ref{p:var-gen-upper}.
The first step is to obtain a decorrelation bound on the increments of~$Y$.
\begin{lemma}\label{l:cov-decay-delta}
  Let~$\Delta_n = Y_{n+1} - Y_{n}$, where~$Y$ is the process defined in~\eqref{e:ZepDef}.  
  Then
  \begin{equation}
    |\cov^x( v \cdot \Delta_{m}, v \cdot \Delta_{m + n + 1} )|
      \leq 4 \abs{v}^2
	\sup_{\tilde y \in \T^d}
	    \norm{ p^{\tilde Y}_n( \tilde y, \cdot ) - 1 }_{L^1(\T^d)}
	\sup_{y \in \R^d} \E^y \abs{\Delta_0}^2
      \,.
  \end{equation}
\end{lemma}

\begin{proof}
  Note first for any~$i,j \in \set{1, \dots, d}$,
  \begin{equation}\label{e:tmpCov1}
    \cov^x(  \Delta^i_m, \Delta^j_{m + n + 1} )
      = \int_{\R^{2d}} f(y') p^Y_n( y', z ) g(z) \, dy' \, dz\,, 
  \end{equation}
  where
  \begin{gather}
    f(y') = \int_{y \in \R^d} p^Y_m(x, y) p^Y_1(y, y') (y'_i - y_i - \E^x \Delta_m^i) \, dy
    \,,
    \\
    g(z) = \int_{z' \in \R^d} p^Y_1(z, z') (z'_j - z_j - \E^x \Delta_{m+n+1}^j) \, dz'
    \,.
  \end{gather}
  Clearly
  \begin{equation}
    \int_{\R^d} f(y') \, dy'
      = \E^x \Delta_m^i - \E^x \Delta_m^i
      = 0 
    \,.
  \end{equation}
  Moreover, for any~$k \in \Z^d$ we note
  \begin{align}
    g(z + k)
      &= \int_{z' \in \R^d} p^Y_1(z + k, z') (z'_j - z_j - k_j - \E^x \Delta_{m+n+1}^j) \, dz'
    \\
      &= \int_{z' \in \R^d} p^Y_1(z, z'-k) ( (z'_j - k_j) - z_j - \E^x \Delta_{m+n+1}^j) \, dz'
      = g(z)\,,
  \end{align}
  and so~$g$ is~$\Z^d$ periodic.
  Thus using the identity~\eqref{e:ptildeY} we see
  \begin{align}
    \int_{\R^{2d}} f(y') p^Y_n( y', z ) g(z) \, dy' \, dz
      &= \sum_{k \in \Z^d} \int_{\R^{d} \times Q_0}
	f(y') p^{Y}_n( y', z + k) g( z) \, dy' \, d z
    \\
      &= \int_{\R^{d} \times \T^d} f(y') p^{\tilde Y}_n( \tilde y', \tilde z ) g( \tilde z) \, dy' \, d\tilde z
    \\
      &= \int_{\R^{d} \times \T^d} f(y') (p^{\tilde Y}_n( \tilde y', \tilde z ) - 1) g( \tilde z) \, dy' \, d\tilde z
    \\
      \label{e:tmpCov2}
      &\leq
	\norm{f}_{L^1(\R^d)}
	\sup_{\tilde y' \in \T^d} \norm{p^{\tilde Y}_n( \tilde y',  \cdot) - 1}_{L^1(\T^d)}
	\sup_{\tilde z \in \T^d} \abs{g( \tilde z)}
	\,.
  \end{align}
Here we used the convention that~$\tilde y = y \pmod{\Z^d}$ denotes the projection of~$y \in \R^d$ to the point~$\tilde y \in \T^d$.

Observe that
\[
\norm{f}_{L^1(\R^d)} \leq  \E^{x}\left[ \left| \Delta_m^i  - \E^{x}[\Delta_m^i] \right| \right]
\]
and
\[
\sup_{ z \in \T^d} \abs{g( z)} \leq \sup_{z \in \T^d} \E^{z}\left[ \left| \Delta_0^j  - \E^{x}[\Delta_{m+n+1}^j] \right| \right]
\]
For any integer $k \geq 1$, the Markov property implies
  \begin{equation}
    \E^{x} \abs{\Delta_k^i}
      = \int_{\R^d} p^Y_k(x, y) \E^y \abs{Y_1 - y} \, dy
      \leq \sup_{y \in \R^d} \E^y \abs{\Delta^i_0}
    \,.
  \end{equation}
  Combining this with~\eqref{e:tmpCov1} and~\eqref{e:tmpCov2} gives
  \begin{equation}
    \abs{\cov^x(  \Delta^i_m, \Delta^j_{m + n + 1} )}
      \leq
	4 \sup_{y \in \R^d} \E^y \abs{\Delta_0^i}
	  \sup_{y \in \R^d} \E^y \abs{\Delta_0^j}
	\sup_{\tilde y' \in \T^d} \norm{p^{\tilde Y}_n( \tilde y', \tilde z ) - 1}_{L^1(\T^d)}
  \end{equation}
  The lemma now follows from the Cauchy--Schwarz inequality.
\end{proof}

Proposition~\ref{p:var-gen-upper} follows quickly from Lemma~\ref{l:cov-decay-delta}.
\begin{proof}[Proof of Proposition~\ref{p:var-gen-upper}]
  Let~$\Delta_k = Y^\epsilon_{k+1} - Y^\epsilon_{k}$, and note
  \begin{align}
    \var^x( v \cdot Y_{N} )
      &= \var^x\paren[\Big]{
	  \sum_{n=0}^{N-1} v \cdot \Delta_n
	}
    \\
      &= 
	\sum_{n=0}^{N-1} \var^x(v \cdot \Delta_n)
	+ 2 \sum_{n = 0}^{N-1}
	  \sum_{k=0}^{N -n - 1}
	    \cov^x(
	      v\cdot \Delta_{n},
	      v\cdot \Delta_{n+k+1}
	    )
    \\
    \label{e:varYtmp1}
      &\leq
	N \abs{v}^2 \sup_{y \in \R^d} \E^y \abs{\Delta_0}^2
	\paren[\Big]{
	  1
	  + C \sum_{k=0}^\infty
	    \sup_{\tilde y \in \T^d} \norm{p_k( \tilde y, \cdot ) - 1 }_{L^1(\T^d)}
	}
    \,.
  \end{align}
  Since~$T = \tmix(\tilde Y)$, for every~$\tilde y \in \T^d$, $n \in \N$ and~$j \in \set{0, \dots, T-1}$, we have
  \begin{equation}\label{e:pnYtildeGeom}
    \norm{p_{nT + j}^{\tilde Y}(\tilde y, \cdot) - 1}_{L^1(\T^d)}
      \leq \frac{1}{2^n}
      \,.
  \end{equation}
  Thus the series on the right hand side of~\eqref{e:varYtmp1} is bounded by~$2T$.
  This implies
  \begin{equation}
    \var^x( v \cdot Y_{N} )
      \leq C N \abs{v}^2 \tmix(Y) \sup_{y \in \R^d} \E^y \abs{\Delta_0}^2
      \,,
  \end{equation}
  which immediately yields~\eqref{e:var-gen-upper}.
\end{proof}
\section{Proof of the lower bound (Theorem~\ref{t:residual-diffusivity})}\label{s:lower}

The proof of Theorem~\ref{t:residual-diffusivity} can be broken down into two steps.
We first prove a general result obtaining a lower bound for the effective diffusivity, provided there is an increasing sequence of stopping times at which the transition density of the stopped process satisfies a Doeblin minorization condition.
Building on our analysis in \cite{IyerLuEA24}, we will then show that the expanding Bernoulli maps we consider satisfy this condition.
Numerical simulations indicate that other systems may also satisfy this Doeblin condition; however, we are not presently able to prove this rigorously.

\subsection{A general lower bound on the asymptotic variance}\label{s:var-lower-general}

Our aim in this section is to consider general Markov processes~$Y$ whose transition density is invariant under~$\Z^d$ shifts as we did in Section~\ref{s:variance-upper}.
The main result in this section obtains a lower bound on the asymptotic variance of~$Y$, provided the transition density is minorized by a distribution that is constant on each of the unit cubes $\set{Q_n}$.
To state this precisely let $w \colon  \R^d \times \Z^d \to [0, 1]$ be a
function such that for every $y \in \R^d$, $w(y, \cdot)$ is a probability distribution on~$\Z^d$.
Moreover, suppose~$w$ is invariant under $\Z^d$-shifts in the sense that for every $y \in \R^d$, $m,n \in \Z^d$ we have
\begin{gather}
  \label{e:w-periodic}
  w(y + n, m + n ) = w(y, m )
  \quad\text{and}\quad
  \sum_{k \in \Z^d} w(y, k) = 1
  \,.
\end{gather}

For every~$y \in \R^d$,  let $\mathcal D_{w(y)}$ be the covariance matrix of the distribution~$w(y, \cdot)$.
That is, let~$\mathcal D_{w(y)}$ be the positive semi-definite matrix whose~$i, j$-th entry is given by
\begin{equation}\label{e:deff-w}
  (\mathcal D_{w(y)})_{i,j} \defeq \cov( \e_i \cdot K, \e_j \cdot K )\,,
\end{equation}
where $K$ is a $\Z^d$ valued random variable for which $\P(K = k ) = w(y, k)$, and $\e_i \in \R^d$ is the $i^\text{th}$ elementary basis vector.
Also define the average covariance matrix~$\bar{\mathcal D}_w$ by
\begin{equation}
  \bar{\mathcal D}_w \defeq \int_{Q_0} \mathcal D_{w(y)} \, dy
  \,.
\end{equation}
We will now show that if the transition density of~$Y$ is minorized by~$w$, then so is the effective diffusivity.

\begin{proposition}\label{p:KV-variance}
  Let $Y$ be a Markov process with finite variance whose transition density~$p^Y$ satisfies~\eqref{e:shift-invariant} and~\eqref{e:leb-inv}.
  If there exists~$\beta > 0$ such that
  \begin{equation}\label{e:doeblin}
    p^Y(y, z) \geq \beta w(y, \floor{z})
    \qquad\text{for every } y,z \in \R^d
    \,,
  \end{equation}
  then
  \begin{equation}
    \lim_{n \to \infty} \frac{ \var(v \cdot Y_{n} ) }{n}
      \geq \beta v \cdot \bar{\mathcal D}_w v
      \,.
  \end{equation}
  Here, we recall $\floor{z}$ is the unique element in~$\Z^d$ such that $z \in \floor{z} + Q_0$.
\end{proposition}

In general, we expect to apply Proposition~\ref{p:KV-variance} as follows.
Let~$\tilde X^\epsilon = X^\epsilon \pmod{\Z^d}$ be the projection of~$X^\epsilon$ to the torus.
If~$T_\epsilon = \tmix^\infty(X^\epsilon)$ denotes the uniform mixing time of~$\tilde X^\epsilon$, then one might expect that the Markov process~$Y_n = X^\epsilon_{n T_\epsilon}$ satisfies the minorization condition~\eqref{e:doeblin}.
Applying Proposition~\ref{p:KV-variance} would now get a lower bound on the effective diffusivity of~$X^\epsilon$ and show
\begin{equation}\label{e:varXn-ep-heuristic}
  \lim_{n \to \infty} \frac{\var(v \cdot X^\epsilon_n)}{n}
    \geq \frac{c}{T_\epsilon} v \cdot \bar{\mathcal D}_{w^\epsilon} v
    \,,
\end{equation}
where
\begin{equation}
  w^\epsilon(m, n) = \P( X^\epsilon_{T_\epsilon} \in Q_n \given X^\epsilon_0 \sim \unif(Q_m) )
  \,.
\end{equation}
One would now need to to estimate the right hand side of~\eqref{e:varXn-ep-heuristic} and show that it does not vanish as~$\epsilon \to 0$.

Both the minorization condition~\eqref{e:p2lower-yhat} for~$Y_n = X^\epsilon_{n T_\epsilon}$, and an~$\epsilon$-independent lower bound for~\eqref{e:varXn-ep-heuristic} are not easy to prove in general.
We will prove Theorem~\ref{t:residual-diffusivity} by proving a version of the above when~$T_\epsilon$ is a time that depends on the starting point.
We will then require a slightly stronger version of Proposition~\ref{p:KV-variance}, which we now state.

\begin{corollary}\label{c:KV-stopped}
  Let~$Y$ be as in Proposition~\ref{p:KV-variance}, and assume the same minorization condition~\eqref{e:doeblin}.
  If $\gamma > 0$ is a constant and~$\tau_n$ is an sequence of bounded stopping times such that $\tau_n \geq \gamma n$ almost surely for all $n \in \N$, then
  then
  \begin{equation}\label{e:KV-stopped}
    \liminf_{n \to \infty} \frac{ \var(v \cdot Y_{\tau_n} ) }{n}
      \geq \beta \gamma v \cdot \bar{\mathcal D}_w v
      \,.
  \end{equation}
\end{corollary}

We prove Proposition~\ref{p:KV-variance} and Corollary~\ref{c:KV-stopped} in Section~\ref{s:KV}, below.

\subsection{Minorizing by the transition density.}

Our aim in this section is to show that the process~$X^\epsilon$ defined in~\eqref{e:Xn} can be time-changed to satisfy the assumptions	of Corollary~\ref{c:KV-stopped}, 
with a constant $\beta$ that is independent of $\epsilon$. 
While this minorization condition is intuitive and pictorially evident, proving it rigorously is somewhat technical and takes up the bulk of this paper.
\smallskip

\begin{figure}[hbt]
  \includegraphics[width=.3\linewidth]{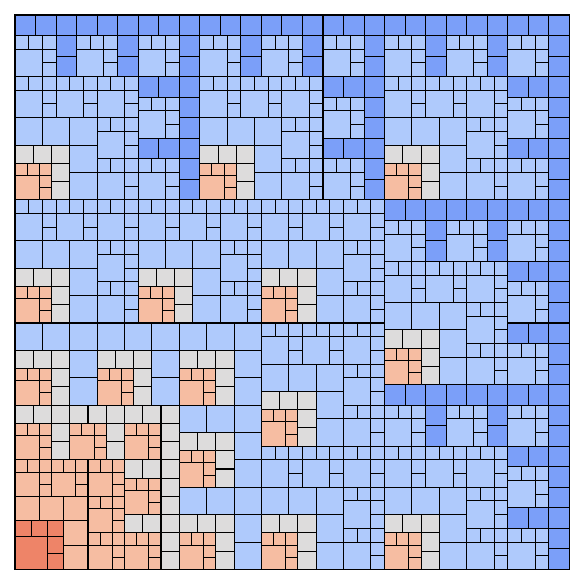}
  \quad
  \includegraphics[width=.3\linewidth]{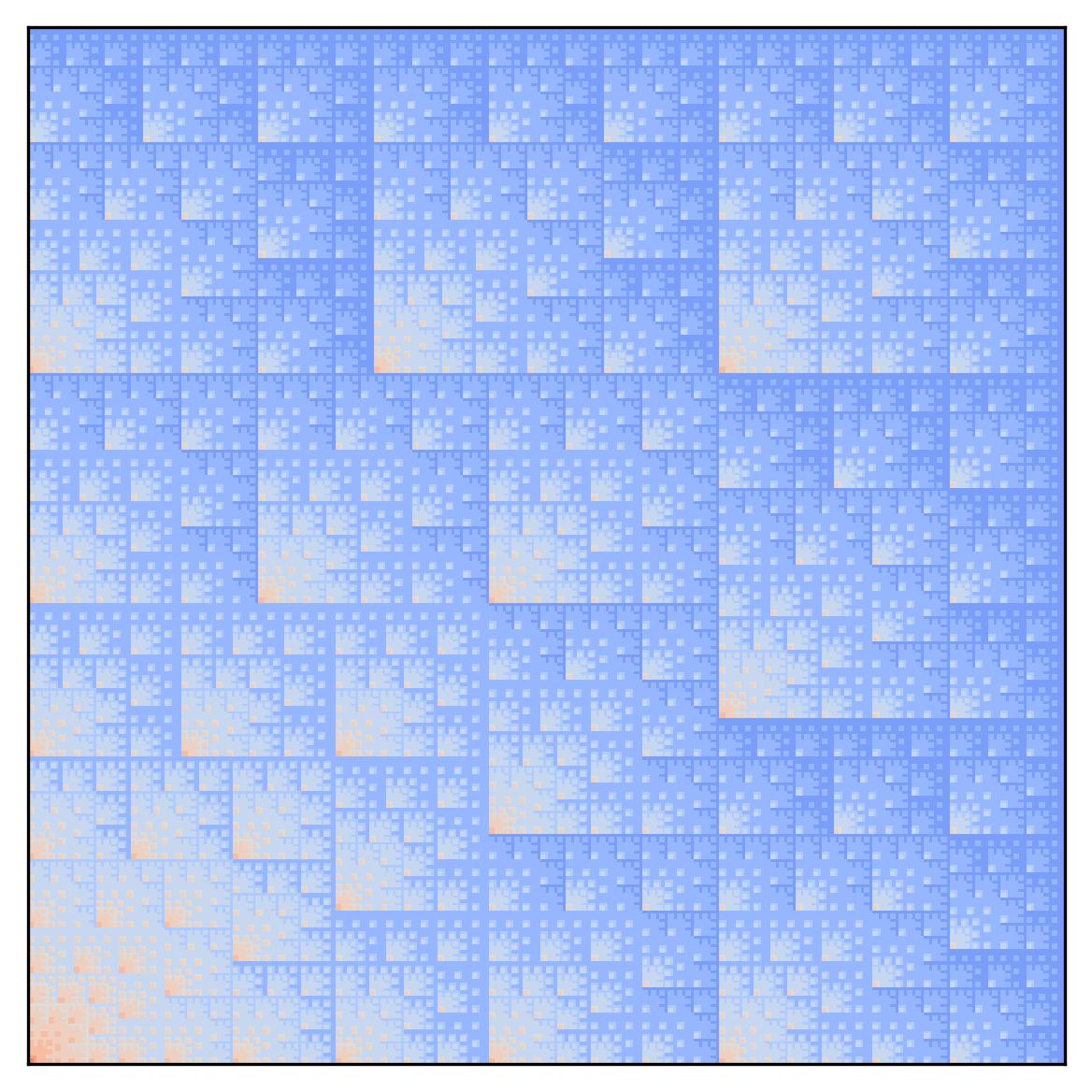}
  \quad
  \includegraphics[width=.3\linewidth]{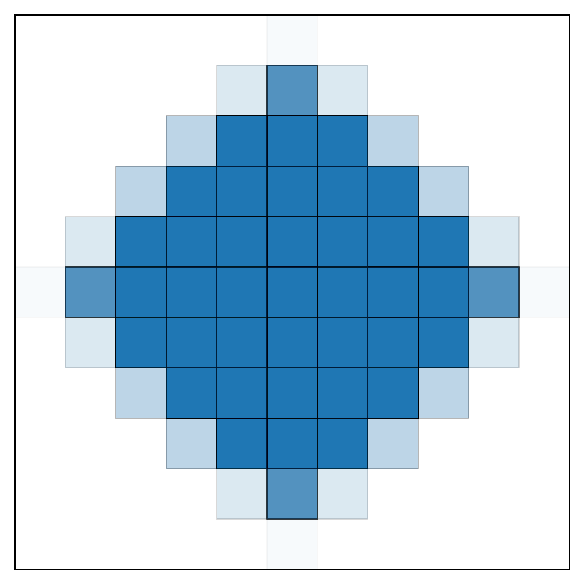}
  \caption{%
    A plot of $\theta^\epsilon$ for two values of~$\epsilon$ (left, center), and the distribution of $Y^\epsilon_1$ given $Y^\epsilon_0 \sim \unif(Q_0)$ (right).
  }
  \label{f:Nepsilon}
\end{figure}

Since~$\varphi$ is constructed from an expanding Bernoulli map, every application of~$\varphi$ induces an expansion.
Because the random perturbations in \eqref{e:Xn} have length scale~$\epsilon$, an important time scale in the dynamics is the time required for sets at the $O(\epsilon)$ scale to be expanded to a set at the $O(1)$ scale.
The expansion time,  however, is not constant as different points expand at different rates.
We will divide~$\R^d$ into cubes of side length~$O(\epsilon)$  (see Figure~\ref{f:Nepsilon}, left) and consider the number of iterations of~$\varphi$  required to expand each of these cubes to the unit cube.
Explicitly, define the (deterministic, $\Z^d$-periodic) function~$\theta^\epsilon \colon \R^d \to \Z^d$ by
\begin{equation}\label{e:thetaDef}
\theta^\epsilon(x) =
    \min\set[\Big]{ m \in \N \st
    \prod_{k = 1}^{m} \abs{\det D\varphi( X^0_k(x))}
      \geq \frac{1}{\epsilon^{d}}
    }
    \,.
\end{equation}

If we start~\eqref{e:Xn} with a delta distribution, the noise will minorize it on cubes of side length~$O(\epsilon)$, and after~$\theta^\epsilon$ steps it is reasonable to expect that the distribution can be minorized on unit cubes.
Unfortunately, the noise also ``leaks mass'' through the boundary making it hard to obtain minorization estimates with a \emph{uniform} lower bound on cubes. We will, instead, minorize the density by a bump distribution supported on unit cubes.
Explicitly, we will show
\begin{equation}\label{e:bump-lower}
  p^{X^\epsilon}_{1 + \theta^\epsilon(x)} (x, y) \geq 
	\chi
	\bm F_*(y)
	\one_{Q(X^0_{1+\theta^\epsilon}(x))}(y)
  \,,
\end{equation}
where~$\chi > 0$ is a constant that is independent of~$\epsilon, x, y$ and~$\bm F_*$ is a specific (explicit) periodic function defined in~\eqref{e:F-star-def} below.
We clarify that~$X^0_{\theta^\epsilon}(x)$ above denotes $X^0_{\theta^\epsilon(x)}(x)$, the deterministic dynamical system~$X^0$ run for time~$\theta^\epsilon(x)$.

Now, when we apply~$\varphi$ to the distribution~$\one_{Q_k} \bm F_*$, it may fragment over many cubes and we are not guaranteed a uniform minorization condition in the form of~\eqref{e:doeblin}.
We thus introduce an auxiliary Markov process~$Z$ to de-fragment the density by solving~\eqref{e:Xn} for~$2 + \theta^\epsilon$ steps, and subtracting a~$\Z^d$ valued ``defragmenting'' shift.

To make this precise, choose an i.i.d.~$\Z^d$ valued sequence of random variables~$I_1, I_2, \dots$ such that
\begin{equation}\label{e:I1-def}
  \P( I_n = \sigma_0(i) ) = \abs{E_i}
  \quad\text{for all } i \in \set{1, \dots, M}
  \,,
  n \in \N
  \,,
\end{equation}
where we recall~$\sigma_0$ is defined in Assumption~\ref{a:phi}.
Next, define an increasing sequence of stopping times~$N^\epsilon_k$ along which the cumulative expansion factor is roughly constant.
Explicitly, set
\begin{equation}
  N_1^\epsilon = 2 + \theta^\epsilon(X_0^\epsilon)
  \,,
\end{equation}
and inductively define
\begin{equation}\label{e:NepDef}
  N_{k+1}^\epsilon
    = N_k^\epsilon + 2 + \theta^\epsilon(X_{N_k^\epsilon}^\epsilon)
    \,.
\end{equation}
Finally define our auxiliary process~$Z^\epsilon$ inductively starting with~$Z^\epsilon_0 = X^\epsilon_0$.
Then, given~$Z^\epsilon_n$, set
\begin{equation}
  \hat Z^\epsilon_{n, N^\epsilon_n} = Z^\epsilon_n
  \,,
  \qquad
  \hat Z^\epsilon_{n, k+1} = \varphi(\hat Z^\epsilon_{n, k}) + \epsilon \xi_k
  \quad\text{for } k > N^\epsilon_n
  \,,
\end{equation}
and define
\begin{equation}\label{e:ZepDef}
  Z^\epsilon_{n+1}
    = \hat Z^\epsilon_{n, N^\epsilon_{n+1}} - I_{n+1}
    \,.
\end{equation}
Since~$\varphi$ has a periodic displacement it is easy to see that
\begin{equation}\label{e:ZMinusX}
  Z^\epsilon_n - X^\epsilon_{N^\epsilon_n} = \sum_{k=1}^n I_k \in \Z^d\,.
\end{equation}
Since~$\theta^\epsilon$ is~$\Z^d$ periodic, equation~\eqref{e:NepDef} implies
\begin{equation}
  N^\epsilon_{k+1} = N^\epsilon_k + 2 + \theta^\epsilon( Z^\epsilon_k )
  \,.
\end{equation}
This implies~$Z^\epsilon$ is a Markov process, and we will now show that its transition density can be minorized on cubes.

Using~\eqref{e:bump-lower} and the definition of $Z^\epsilon$ we will show (see Lemma~\ref{l:p1z-doeblin}, below) that the density $z \mapsto p^{Z^\epsilon}_1(x,y)$ of $Z^\epsilon$ after one step is bounded below by a constant over the cube $Q(X^0_{1+ \theta^\epsilon}(x))$.
If we directly use this along with Proposition~\ref{p:KV-variance} we only get an~$O(1)$ lower bound for the asymptotic variance of~$Z^\epsilon$.
This isn't sufficient to complete the proof of Theorem~\ref{t:residual-diffusivity}.
However, if we iterate this process once, and obtain a lower bound on the two step transition density of~$Z^\epsilon$, we obtain a lower bound that is sufficient to complete the proof of Theorem~\ref{t:residual-diffusivity}.
Precisely, we will prove the following lower bound.

\begin{lemma}\label{l:p2z-doeblin}
  There exists a constant $c > 0$ such that the two step transition density of~$Z^\epsilon$ satisfies
  \begin{equation}\label{e:p2lower-yhat}
    p_2^{Z^\epsilon}(x, y)
      \geq c \check w^\epsilon( X^0_{1 + \theta^\epsilon}(x), \floor{y} )\,, \quad
      \text{for all } x,y \in \R^d
    \,.
  \end{equation}
  Here~$\check w^\epsilon$ is defined by
  \begin{equation}\label{e:wEpCheckDef}
    \check w^\epsilon(z, k)
      \defeq \P( X^0_{1 + \theta^\epsilon}(U) \in Q_k \given U \sim \unif(Q_{\floor{z}} ), \quad z \in \mathbb{R}^d, \; k \in \mathbb{Z}^d
    \,.
  \end{equation}
\end{lemma}

Because $\check w^\epsilon$ depends on $z$ only via the integer part $\floor{z}$, the covariance matrix $\mathcal{D}_{\check w^\epsilon(x)}$ does not depend on $x$: $\bar{\mathcal D}_{\check w^\epsilon} = \mathcal{D}_{\check w^\epsilon(x)}$ for all $x$, so we write simply $\mathcal{D}_{\check w^\epsilon}$. Proposition~\ref{p:KV-variance} can now be used to show estimate the asymptotic variance of~$Z^\epsilon$ in terms of the diffusivity matrix~$\bar{\mathcal D}_{\check w^\epsilon}$.   This matrix can be computed explicitly as an~$\epsilon$-independent matrix times an~$O(\abs{\ln \epsilon})$ factor.
This is our next result.

\begin{lemma}\label{l:DwCheck}
Let $w^\epsilon(x,y) = \check w^\epsilon( X^0_{1 + \theta^\epsilon}(x), \floor{y} )$.  The covariance of $y \mapsto w^\epsilon(x,y)$ does not depend on $x$, and for all $\epsilon > 0$, this covariance matrix $\mathcal{D}_{w^\epsilon}$ is given by
\begin{equation}\label{e:DwCheck1}
    \mathcal{D}_{w^\epsilon}
      = (1 + \bar \theta^\epsilon) \mathcal D_{w^0}
  \end{equation}
  where  
  \begin{gather}
    \bar \theta^\epsilon
      \defeq \int_{y \in Q_0} \theta^\epsilon(y)  \, dy
    \,,
    \\
    w^0(m, k)
      \defeq \P( \varphi(U) \in Q_k \given U \sim \unif(Q_m) ), \quad m,k \in \Z^d
    \,.
  \end{gather}

\end{lemma}

Lemmas~\ref{l:p2z-doeblin} and~\ref{l:DwCheck} will be proved in Section~\ref{s:minorizing}, below.

\subsection{Proof of Theorem \ref{t:residual-diffusivity}}

We will now prove Theorem~\ref{t:residual-diffusivity}.
Proposition~\ref{p:KV-variance} and Lemma~\ref{l:p2z-doeblin} immediately give a lower bound on the asymptotic variance of~$Z^\epsilon$.
To obtain the asymptotic variance of~$X^\epsilon$ from~$Z^\epsilon$, we need to undo the shifts by the process~$I$, and time change by~$N^\epsilon$.
We carry out the details here.

\begin{proof}[Proof of Theorem~\ref{t:residual-diffusivity}]
  Here and subsequently we use the convention that~$C < \infty$ is a large $\epsilon$-independent constant whose value is unimportant and may increase from line to line.
  We also use $c > 0$ to denote a small $\epsilon$-independent constant whose value is unimportant and may decrease from line to line.
  Define the family of stopping times~$M^\epsilon_n$ by
  \begin{equation}
    M^\epsilon_n = \inf \set{ k \in 2 \N \st N^\epsilon_{k} \geq \abs{\ln \epsilon} n }, \quad \quad n=1,2,3,\dots
    \,.
  \end{equation}
  We claim 
  \begin{equation}\label{e:MepBounds}
    c n \leq M^\epsilon_n \leq C n, 
     \quad \quad \forall \; n=1,2,3,\dots.
  \end{equation}
  To see this, note
  \begin{equation}\label{e:dPhiBound}
    \frac{1}{\abs{E_M}} \leq \abs{ \det D\varphi} \leq \frac{1}{\abs{E_1}}\,,
  \end{equation}
  which implies
  \begin{equation}\label{e:theta-bound}
    \frac{d \ln \epsilon }{\ln \abs{E_1} }
      \leq \theta^\epsilon(x)
      < \frac{d \ln \epsilon }{\ln \abs{E_M} } + 1
    \,,
  \end{equation}
  for every~$x \in \R^d$.
  This in turn implies
  \begin{equation}\label{e:Nep1}
    n \paren[\Big]{2 + \frac{d \ln \epsilon }{\ln \abs{E_1} }}
      \leq N^\epsilon_n
      < n\paren[\Big]{ 3 + \frac{d \ln \epsilon }{\ln \abs{E_M} }} 
  \end{equation}
  and immediately yields~\eqref{e:MepBounds} as claimed.

Now, consider the Markov chain $Z^\epsilon_{2n}$.
By Lemma~\ref{l:p2z-doeblin}, the transition density for $Z^\epsilon_{2n}$ satisfies the minorization condition required in Corollary~\ref{c:KV-stopped}, with $w^\epsilon(x,k) = \check w^\epsilon( X^0_{1 + \theta^\epsilon}(x), k )$. 
By~\eqref{e:MepBounds}, $M^\epsilon$ is a bounded sequence of stopping times, and so Corollary~\ref{c:KV-stopped} applied to the process $Z^\epsilon_{2n}$ gives
  \begin{equation}\label{e:avZep-lower}
    \lim_{n \to \infty} \frac{1}{n} \var(v \cdot Z^\epsilon_{M^\epsilon_n})
      \geq c v \bar{\mathcal D}_{\check w^\epsilon} v
      = c (1 + \bar \theta^\epsilon) v \mathcal D_{w^0} v
      \,,
  \end{equation} 
  where the last equality above followed from~\eqref{e:DwCheck1}.  

  Next we note that~\eqref{e:ZMinusX} and independence of~$I$ and~$Y$ imply 
  \begin{equation}
    \var(v \cdot X^\epsilon_{N^\epsilon_{M^\epsilon_n}})
      = \var(v \cdot Z^\epsilon_{M^\epsilon_n})
	+ \E M^\epsilon_n \var(v \cdot I_1)
      \,.
  \end{equation}
  Combining this with~\eqref{e:MepBounds} and~\eqref{e:avZep-lower} implies
  \begin{equation}\label{e:varYep}
    \lim_{n \to \infty} \frac{1}{n} \var(v \cdot X^\epsilon_{N^\epsilon_{M^\epsilon_n}})
      \geq c (1 + \bar \theta^\epsilon) v \mathcal D_{w^0} v - C \abs{v}^2
      \,.
  \end{equation} 

  The above computes variance of $X^\epsilon$ evaluated at a sequence of stopping times~$N^\epsilon_{M^\epsilon_n}$.
  In order to compute the variance of $X^\epsilon$ along a sequence of deterministic times we use~\eqref{e:theta-bound} to note
  \begin{equation}\label{e:Nep2}
    \frac{d\ln \epsilon }{\ln \abs{E_1} } + 1
      < N^\epsilon_{n+1} - N^\epsilon_n
      \leq \frac{d\ln \epsilon }{\ln \abs{E_M} } + 2
    \,.
  \end{equation}
  This implies
  \begin{equation}\label{e:nep-NM}
    n^\epsilon
      \leq N^\epsilon_{M^\epsilon_n}
      \leq n^\epsilon + \frac{2 d \ln\epsilon}{\ln \abs{E_M}} + 2
    \qquad\text{where}\quad
    n^\epsilon \defeq \floor{n \abs{\ln \epsilon}}
    \,,
  \end{equation}
  showing $N^\epsilon_{M^\epsilon_n}$ differs from the deterministic time~$n^\epsilon$ by an $O(\abs{\ln \epsilon})$ time that does not grow with~$n$.

  This allows us to estimate $\var(X^\epsilon_{n^\epsilon})$ as follows.
  First note,
  \begin{equation}\label{e:tmp-2024-03-02-2}
    \var(v \cdot X^\epsilon_{n^\epsilon})
      \geq \frac{1}{2} \var( v \cdot X^\epsilon_{N^\epsilon_{M^\epsilon_n}} )
	- 2 \var( v \cdot (X^\epsilon_{N^\epsilon_{M^\epsilon_n}} - X^\epsilon_{n^\epsilon} ))
    \,.
  \end{equation}
  Next we note~\eqref{e:Xn} implies
  \begin{equation}
    \abs{X^\epsilon_{n+1} - X^\epsilon_n}
      \leq \sup_{x \in \R^d} \abs{\varphi(x) - x} + \epsilon \abs{\xi_{n+1}}
  \end{equation}
  and hence
  \begin{equation}
\abs{X^\epsilon_{N^\epsilon_{M^\epsilon_n}}
	- X^\epsilon_{n^\epsilon}}
      \leq C (N^\epsilon_{M^\epsilon_n} - n^\epsilon)
	+ \epsilon \sum_{k = n^\epsilon}^{N^\epsilon_{M^\epsilon_n}} \abs{\xi_k}
    \,.
  \end{equation} 
  Using~\eqref{e:nep-NM} this implies
  \begin{equation}\label{e:tmp-2024-03-02-1}
    \E \abs{X^\epsilon_{N^\epsilon_{M^\epsilon_n}} - X^\epsilon_{n^\epsilon}}^2
      \leq C \abs{\ln \epsilon}^2
    \,.
  \end{equation}
  Substituting~\eqref{e:tmp-2024-03-02-1} in~\eqref{e:tmp-2024-03-02-2} gives
  \begin{equation}
    \var(v \cdot X^\epsilon_{n^\epsilon})
      \geq \frac{1}{2} \var( v \cdot X^\epsilon_{N^\epsilon_{M^\epsilon_n}})
	- C\abs{v}^2 \abs{\ln \epsilon}^2
    \,.
  \end{equation}

  Hence
  \begin{align}
    \lim_{n \to \infty} \frac{\var(X^\epsilon_n)}{n}
      &= \lim_{n \to \infty} \frac{\var(X^\epsilon_{n^\epsilon})}{n^\epsilon}
      \geq \lim_{n \to \infty}
	\frac{1}{n^\epsilon} \paren[\Big]{
	  \frac{1}{2} \var( v \cdot X^\epsilon_{N^\epsilon_{M^\epsilon_n}} )
	    - C \abs{v}^2 \abs{\ln \epsilon}^2
	}
    \\
      &= \frac{1}{2}
	  \lim_{n \to \infty}
	    \frac{\var( v \cdot X^\epsilon_{N^\epsilon_{M^\epsilon_n}} )}{n}
	  \lim_{n \to \infty}
	    \frac{n}{n^\epsilon}
    \\
      &\geq \frac{1}{\abs{\ln\epsilon}}
      \paren[\big]{
	c (1 + \bar\theta^\epsilon) v \mathcal D_{w^0} v
	- C \abs{v}^2
      }
    \,,
  \end{align}
  where the last inequality above followed from~\eqref{e:varYep}.
  Using~\eqref{e:theta-bound} this implies
  \begin{equation}
    \lim_{n \to \infty} \frac{\var(X^\epsilon_n)}{n}
      \geq c v \mathcal D_{w^0} v
	- \frac{C \abs{v}^2}{\abs{\ln\epsilon}}
      \,,
  \end{equation}
  which immediately implies~\eqref{e:residual-diffusivity} and concludes the proof.
\end{proof}

\section{A lower bound on the asymptotic variance (Proposition \ref{p:KV-variance})}\label{s:KV}
In this section we prove Proposition~\ref{p:KV-variance} using an idea that can be traced back to Kipnis and Varadhan~\cite{KipnisVaradhan86} and has been widely used in the analysis of markov processes~\cite{KomorowskiLandimEA12}. 
The result of Kipnis and Varadhan~\cite{KipnisVaradhan86} gives an explicit formula for the asymptotic variance, in terms of a corrector.
In the context of~\eqref{e:Xn} when~$\varphi$ is chaotic, the corrector is not well behaved and it is not easy to obtain enough bounds on the corrector which are good enough to yield the lower bound in Theorem~\ref{t:residual-diffusivity}.
We will, instead, show that if a minorizing condition holds then one can obtain bounds on the asymptotic variance \emph{without} requiring any bounds on the corrector.

Let~$Y$ be a Markov process that satisfies the assumptions in Proposition~\ref{p:KV-variance}.  The strategy is to write $Y_n$ as $M_n + n \bar s - \chi(Y_n)$ where each component of the vector-valued process $M_n$ is a martingale, $\bar s$ is a constant drift, and $\chi(y)$ is a bounded function.  Let~$\mathcal L$ be the operator
\begin{equation}
  \mathcal Lf(y)
    = \E^y (f(Y_1)  - f(y))
    = \int_{\R^d} p^Y_1( y, z) (f(z) - f(y)) \, dz
    \,.
\end{equation}
In view of the assumptions on the density $p$, this is well-defined for any measureable function $f$ that satisfies a bound of the form $|f(y)| \leq C (1 + |y|^2)$, for $y \in \R^d$. 

\begin{lemma}\label{l:zeta}
  There exists a bounded, $\Z^d$-periodic function $\chi \colon \R^d \to \R^d$ and a vector $\bar s \in \R^d$ such that the function~$\zeta$ defined by
  \begin{equation}\label{e:zeta}
    \zeta(y) \defeq y + \chi(y)
  \end{equation}
  satisfies
  \begin{gather}\label{e:Lf-equals-D}
    \mathcal L \zeta  = \bar s \,.
  \end{gather}
\end{lemma}
\begin{remark}
  Here $\zeta(y) = (\zeta_1(y),\dots,\zeta_d(y)) \in \R^d$,
  and so the equation $\mathcal L \zeta  = \bar s$ means $\mathcal L \zeta_i  = \bar s_i$ holds for every $i = 1,\dots,d$.
\end{remark}

Lemma~\ref{l:zeta} can be proved by writing down an explicit series representation for~$\zeta$.
To avoid distracting from the main idea, we momentarily postpone the proof.
For any $v \in \R^d$, define the function $V_v \colon \R^d \to \R$ by
\begin{equation}
  V_v(y) \defeq
    \E^y (v \cdot (\zeta(Y_1) - \zeta(Y_0) - \bar s) )^2
    = \var( v \cdot \zeta(Y_1) \given Y_0 = y  )
    \,.
\end{equation}
Since the transition density of~$Y$ is invariant under $\Z^d$ shifts~\eqref{e:shift-invariant}, the function $V_v$ is $\Z^d$-periodic.
We will now show that the asymptotic variance of~$v \cdot Y$ is exactly the integral of~$V_v$.

\begin{theorem}[Kipnis, Varadhan~\cite{KipnisVaradhan86}]\label{t:kipnis-varadhan}
  For any~$v \in \R^d$ we have
  \begin{equation}\label{e:var-v.Yn}
    \lim_{n \to \infty} \frac{ \var(v \cdot Y_n) }{n}
      = \int_{Q_0} V_v(y) \, dy
      \,.
  \end{equation}
\end{theorem}
\begin{proof}
  We first show that for all $n \in \N$ and $v \in \R^d$ we have
  \begin{equation}\label{e:KV-variance}
    \var( v \cdot \zeta(Y_n) )
      = \sum_{k=0}^{n-1} \E V_v( Y_k )
    \,.
  \end{equation}
  To see this, note that
\[
\E^y (\zeta(Y_1) - \zeta(y)) = \mathcal{L}\zeta(y) = \bar s, \quad y \in \R^d.
\]
Hence, if $\mathcal{F}_n$ is the filtration generated by $Y_0,Y_1,\dots,Y_n$, we have
  \begin{equation}
    \E [\zeta(Y_{n+1}) - \zeta(Y_n)\;|\; \mathcal{F}_n ]     
      =  \mathcal{L}\zeta(Y_{n}) = \bar s\,.
  \end{equation}
  As a result, each component of the vector-valued process
  \begin{equation}\label{e:Mn}
    M_n \defeq \zeta(Y_n) - n \bar s
  \end{equation}
  is a martingale.
  Thus 
  \begin{align}
\var( v \cdot (\zeta(Y_n) - \zeta(Y_0)) )
      &= \E( v \cdot (M_n - M_0)^2 )
      = \sum_{k = 0}^{n-1} \E  (v \cdot (M_{k+1} - M_k) )^2
    \\
      &= \sum_{k = 0}^{n-1} \E  (v \cdot (\zeta(Y_{k+1}) - \zeta(Y_k) - \bar s) )^2
    \\
    \label{e:tmp-2024-03-08-2}
      &= \sum_{k = 0}^{n-1} \E  \E^{Y_k} (v \cdot (\zeta(Y_{1}) - \zeta(Y_0) - \bar s) )^2
      = \sum_{k = 0}^{n-1} \E  V_v(Y_k)
      \,,
  \end{align} 
  proving~\eqref{e:KV-variance}.

  To see that~\eqref{e:KV-variance} implies~\eqref{e:var-v.Yn}, recall that $y = \zeta(y) - \chi(y)$, where $\chi$ is periodic and bounded.  Hence,  we have
  \begin{align}
       \lim_{n \to \infty} \frac{ \var(v \cdot Y_n )}{n}  & =  \lim_{n \to \infty} \frac{ \var(v \cdot (Y_n-Y_0) )}{n} \\
      & = \lim_{n \to \infty} \frac{ \var(v \cdot (\zeta(Y_n) - \zeta(Y_0) -  \chi(Y_n) + \chi(Y_0)) )}{n}  \\
      & = \lim_{n \to \infty} \frac{ \var(v \cdot (\zeta(Y_n)- \zeta(Y_0))  )}{n} \,.
  \end{align}

Let~$\tilde Y$ denote the projection of~$Y$ to the torus. Since $V_v$ is periodic, we have $V_v(Y_k) =  V_v(\tilde Y_k)$.  The minorizing condition~\eqref{e:doeblin} for $Y$ implies that the projected chain $\tilde Y$ on $\T^d$ is ergodic.  In fact, this condition implies that the mixing time of this chain on $\T^d$ satisfies $\tmix(\tilde Y) \leq 1 / \beta$ (see for instance~\cite{LevinPeres17,MontenegroTetali06}).
Therefore, applying the ergodic theorem to~\eqref{e:KV-variance}, we conclude
  \begin{equation}
    \lim_{n \to \infty} \frac{ \var(v \cdot Y_n )}{n} 
      = \lim_{n \to \infty} \frac{1}{n} \sum_{k=0}^{n-1} \E V_v( \tilde Y_k )
      = \int_{Q_0} V_v(y) \, dy
    \,.
    \qedhere
  \end{equation}
\end{proof}

We will now use Theorem~\ref{t:kipnis-varadhan} to prove Proposition~\ref{p:KV-variance}.
\begin{proof}[Proof of Proposition~\ref{p:KV-variance}]
  By Theorem~\ref{t:kipnis-varadhan} it suffices to show
  \begin{equation}\label{e:tmp-2024-02-09-1}
    \int_{Q_0} V_v(y) \, dy
      \geq \int_{Q_0} \beta v \cdot \mathcal D_{w( y)} v \, d y
      \,,
  \end{equation}
  for all $v \in \R^d$.
  To see this, observe
  \begin{align}
    \MoveEqLeft
    \int_{Q_0} V_v(y) \, dy
      = \int_{y \in Q_0} \E^y( v \cdot (\zeta(Y_1) - \zeta(Y_0) - \bar s) )^2 \, dy
    \\
      &= \int_{y \in Q_0} \E^y( v \cdot (Y_1 + \chi(Y_1) - y - \chi(y) - \bar s) )^2 \, dy
    \\
      &=
	\int_{y \in Q_0} 
	\int_{z \in \R^d}
	  p(y, z)
	  ( v \cdot (z + \chi(z) - y - \chi(y) - \bar s) )^2
	  \, dz \, dy
    \\
      \label{e:tmp-2024-02-07-1}
      &\geq
	\beta
	\int_{y, z \in Q_0} 
	\sum_{k \in \Z^d} w(y, k)
	  ( v \cdot (k + z + \chi(z) - y - \chi(y) - \bar s) )^2
	  \, dy \, d z.
  \end{align}

  Now let $K$ be a $\Z^d$ valued random variable with $\P(K = k) = w(y, k)$.
  Then
  \begin{align}
    v \cdot \mathcal D_{w(y)} v
      &= \var( v \cdot K )
      = \sum_{k \in \R^d}
	  w(y, k) (v \cdot (k - \E K))^2
    \\
      &= \inf_{b \in \R^d} \sum_{k \in \R^d} w( y, k) (v \cdot (k - b))^2
      \,.
  \end{align}
  Hence,~\eqref{e:tmp-2024-02-07-1} implies
  \begin{align}
    \int_{\T^d} V_v(y) \, dy
      &\geq \beta \int_{y, z \in Q_0}
	\sum_{k \in \Z^d} w( y, k)
	  ( v \cdot (k + z + \chi(z) - y - \chi(y) - \bar s) )^2
	  \, dy \, d z
      \\
      &\geq \int_{y \in Q_0} \beta v \cdot \mathcal D_{w(y)} v \, dy
      = \beta v \cdot \bar{\mathcal D}_{w} v
      \,,
  \end{align}
  yielding~\eqref{e:tmp-2024-02-09-1}, and concluding the proof.
\end{proof}

The proof of Corollary~\ref{c:KV-stopped} is very similar to that of Proposition~\ref{p:KV-variance}, and we outline it here.
\begin{proof}[Proof of Corollary~\ref{c:KV-stopped}]
  Let~$M_n$ be the martingale defined in~\eqref{e:Mn}, and use the same argument used to derive~\eqref{e:tmp-2024-03-08-2} to obtain 
  \begin{equation} 
   \var( v \cdot (\zeta(Y_{\tau_n})- \zeta(Y_0)) ) = \E \sum_{k = 0}^{\tau-1} V_v(Y_k)
    \,.
  \end{equation}
  Since $V_v \geq 0$, this implies
  \begin{equation}
   \var( v \cdot (\zeta(Y_{\tau_n})- \zeta(Y_0)) ) \geq \E \sum_{k = 0}^{\gamma n-1} V_v(Y_k)
    \,.
  \end{equation}
  Now following the same argument as in the proof of Proposition~\ref{p:KV-variance} will yield~\eqref{e:KV-stopped} as desired.
\end{proof}

It remains to prove existence of the corrector function~$\chi$, as stated in Lemma~\ref{l:zeta}.
\begin{proof}[Proof of Lemma~\ref{l:zeta}]
  Define the function~$s \colon \R^d \to \R^d$ and the vector~$\bar s \in \R^d$ by
  \begin{equation}
    s(y)
      \defeq \E^{y} (Y_1 - y)
      \,,
    \qquad\text{and}\qquad
    \bar s \defeq \int_{Q_0} s(y) \, d y \,.
\end{equation}
The property~\eqref{e:shift-invariant} implies that $s$ is a periodic function of~$y$.
Although $s(y)$ need not be continuous, it is bounded by
\[
\sup_{y \in \R^d} |s(y)| \leq \sqrt{d} + \sup_{y \in Q_0} \int_{\R^d} p^Y(y,z)|z| \,dz < \infty.
\] 

Our goal is to construct a bounded, periodic function $\chi$ satisfying 
\begin{equation}\label{e:LchiequalsD}
  \mathcal{L} \chi(y)= \bar s - s(y)
  \,.
\end{equation}
We claim that the series
\begin{equation}
\chi(y) = \sum_{n=0}^\infty \E^{\tilde y} (  s(\tilde Y_n) - \bar s) \,,
  \quad y \in \R^d \label{chidef}
\end{equation}
is well-defined and gives the desired function.
Here~$\tilde y \in \T^d$ is the equivalence class~$y \pmod{\Z^d} \in \T^d$, and $\tilde Y_n = Y_n \pmod{\Z^d}$ (as in~\eqref{e:tildeY}) is the projection of~$Y$ onto the torus~$\T^d$.

Let $p^{\tilde Y}$ be the transition density on $\T^d$ for the projected chain $\tilde Y$. Then
\begin{align}
\E^{y} ( s(\tilde Y_n) - \bar s)
  &=\int_{\T^d}p_n^{\tilde Y}( \tilde y, \tilde z) ( s(\tilde z) - \bar s) \,d\tilde z
  = \int_{\T^d}(p_n^{\tilde Y}( \tilde y, \tilde z) - 1) ( s(\tilde z) - \bar s) \,d\tilde z
\\
  &\leq
    \norm{s - \bar s}_{L^\infty(\T^d)}\sup_{\tilde y \in \T^d} \int_{\T^d} |p_n^{\tilde Y}(\tilde y, \tilde z) - 1| \,d \tilde y
    \,.
\end{align}

Now let~$T = \tmix(\tilde Y)$ be the mixing time of~$\tilde Y$.
As in the proof of Theorem~\ref{t:kipnis-varadhan}, the minorizing condition~\eqref{e:doeblin} implies that $T \leq 1 / \beta < \infty$.
Since the right hand side decreases geometrically (exactly as in~\eqref{e:pnYtildeGeom}), the series on the right hand side of~\eqref{chidef} converges and
\begin{equation}
  \norm{\chi}_{L^\infty(\T^d)}
    \leq 
      2 \tmix(\tilde Y)
      \norm{s - \bar s}_{L^\infty(\T^d)}
    \,.
\end{equation}
By shift invariance~\eqref{e:shift-invariant}, the function~$\chi$ is~$\Z^d$ periodic. 

Finally, we check that $\chi$ satisfies \eqref{e:LchiequalsD}.
Notice
  \begin{align}
    \mathcal L\chi(y)
      & =  \E^y \left( \sum_{k=0}^\infty  \E^{Y_1}  (s(Y_k) - \bar s)\right) - \chi(y) \\
      &  =  \left(\sum_{k=0}^\infty \E^y (s(Y_{k+1}) - \bar s)\right) - \chi(y) \\
      & = - \E^y (s(Y_{1}) - \bar s)  = \bar s - s(y)
      \,,
  \end{align}
  which implies~$\chi$ solves~\eqref{e:LchiequalsD} as desired.
  Since~$\zeta$ is defined by~\eqref{e:zeta}, we this implies~$\zeta$ satisfies~\eqref{e:Lf-equals-D}, concluding the proof.
\end{proof}

\section{Minorizing the transition density (Lemmas~\ref{l:p2z-doeblin},~\ref{l:DwCheck})}\label{s:minorizing}

Finally, we conclude this paper by proving the minorization condition for the process~$Z^\epsilon$ (Lemma~\ref{l:p2z-doeblin}), and estimating the effective diffusivity of the minorizing distribution (Lemma~\ref{l:DwCheck}).
The proof of this relies on the Bernoulli structure of~$\varphi$, and requires some notational preparation.

\subsection{Cylinder sets and shifts.}
We begin by setting up our notation for cylinder sets and shifts that will be used in the proof.
Let $\mathcal I = \set{1, \dots, M}$, and $\mathcal T$ denote the set of all finite length $\mathcal I$-valued tuples.
Explicitly,
\begin{equation*}
    \mathcal T
      = \set{\zerotuple} \cup \bigcup_{m = 1}^\infty \mathcal I^m,
\end{equation*}
where $\zerotuple$ denotes the empty tuple.
Given a tuple $s = (s_0, \dots, s_{m-1}) \in \mathcal T$ we use $\abs{s} = m$ to denote the length of the tuple~$s$, with \(\abs{\zerotuple} = 0\) by convention. 

Let \(\tilde \sigma\colon \mathcal T \to \mathcal T\) be the \emph{Bernoulli left shift}.
That is, $\tilde \sigma(s)$ removes the first coordinate of~$s$ and shifts the other coordinates left.
More precisely, we define
\[
    \tilde \sigma(s_0, \ldots, s_{m-1}) = (s_1, \ldots, s_{m-1}) \,,
    \quad\text{and}\quad
    \tilde \sigma(\zerotuple) = \zerotuple\,.
\]
Let~$\tilde \sigma^k$ denote the \(k\)-fold composition of the map \(\tilde \sigma\).
The map~$\tilde \sigma$ is the Bernoulli shift corresponding to the periodic map~$\tilde \varphi \colon \T^d \to \T^d$ induced by~$\varphi$.

Given $k \in \Z^d$ and $s = (s_0, \dots, s_{m-1}) \in \mathcal T$, define
\begin{equation}
  \sigma(k, s) = (k + \sigma_0(s_0), \tilde \sigma(s) )
  \,,
\end{equation}
where~$\sigma_0 \colon \mathcal I \to \Z^d$ is the function defined in Assumption~\ref{a:phi}.
We now define the cylinder set associated to $(k, s) \in \Z^d \times \mathcal T$ by
\begin{equation}\label{e:CylS}
  \cyl_{k, s} \defeq \set{
    x \in Q_k \st
      \varphi^n(x) - \floor{\varphi^n(x)} \in  E_{s_n} \text{ for all } n \leq \abs{s}
    }
    \,.
\end{equation}
When~$s = \zerotuple$, the associated cylinder set $\cyl_{k, s}$ is the cube $Q_k \subseteq \R^d$.
When $s = (s_0)$ is a tuple of length~$1$, the associated cylinder set~$\cyl_{k, s}$ is simply the domain~$k + E_{s_0} \subset Q_k$.
When~$\abs{s} > 1$, each of these cylinder sets get subdivided further, forming finer and finer partitions of one of the cubes $Q_n$.  Observe that for any $s \in \mathcal{T}$, 
\[
\cyl_{k, s} = (k - j)+ \cyl_{j, s}, \quad \forall \; k,j \in \Z^d.
\]
In particular, the volume $|\cyl_{k, s}|$ does not depend on $k$. Recall by Assumption~\ref{a:phi}, all cylinder sets are intervals for~$d = 1$ and axis-aligned cubes for $d > 1$. We will use~$\ell_s$ denote the length of the interval~$\cyl_{k,s}$ when~$d = 1$, and the side length of the cube~$\cyl_{k,s}$ when $d > 1$; $\ell_s$ does not depend on the integer coordinate $k$.  For convenience define~$\lambda_s = 1/\ell_s$.
Explicitly,
\begin{equation}\label{elldef}
  \ell_s = |\cyl_{k,s}|^{1/d}\,,
  \quad\text{and}\quad
  \lambda_s = \frac{1}{\ell_s} = \frac{1}{|\cyl_{k,s}|^{1/d}}\,.
\end{equation}


\subsection{Minorizing by a bump function.}

In this section we will show that the distribution of~$X^\epsilon(x)$ after time~$\theta^\epsilon$ is minorized by a bump function on a cube (as in~\eqref{e:bump-lower}).
If $s \in \mathcal T$ then notice
\begin{equation} \label{e:phicylinder}
  \varphi(\mathring \cyl_{k, s})
    = \begin{dcases}
	\mathring \cyl_{\sigma(k, s)} & s \neq \zerotuple\,,
	\\
	\bigcup_{j \in \mathcal I} Q_{k + \sigma_0(j)} - \mathcal N
	& s= \zerotuple
	\,,
      \end{dcases}
\end{equation}
for some null set~$\mathcal N$.
In particular, this means that if $|s| \geq 1$,
\begin{equation} \label{Jksdef}
\varphi^{\abs{s}}(\mathring \cyl_{k, s}) = \mathring Q_{J(k,s)}, \quad \quad \text{where}\quad J(k,s) = k + \sum_{i=0}^{|s|-1} \sigma_0(s_i) \in \Z^d.
\end{equation}
Thus, an initial distribution that is supported on~$\cyl_{k, s}$ becomes spread over~$Q(J(k, s))$ after~$\abs{s}$ iterations of~$\varphi$.

Since the process~$X^\epsilon$ is constructed by intertwining the action of~$\varphi$ with noise, this suggests that if~$X^\epsilon_0$ is concentrated on one cylinder set~$\cyl_{k,s}$, then the distribution of $X^\epsilon_{\abs{s}}$ should be spread out over some cube~$Q_n$.
This, however, is not easy to prove as the action of the noise does not commute with the dynamics of~$\varphi$.
The main idea introduced in~\cite{IyerLuEA24} is to construct a family of distributions $\bm F_{k, s}$ whose evolution under $X^\epsilon$ is controlled.

For any $k \in \Z^d$, $s \in \mathcal T$, define
\begin{equation}\label{e:FkDef}
  \bm F_{k, s} (x) =
  \begin{dcases}
    \one_{Q_k}(x) \bm F_*(x)
	  & s = \zerotuple
	  \,,
	  \\
	  \frac{1}{\abs{\cyl_{k, s}}}
	  \one_{\cyl_{k,s}} \bm F_*\circ \varphi^{\abs{s}}(x)
	  & s \neq \zerotuple
	  \,,
  \end{dcases}
\end{equation}
where
\begin{equation}\label{e:F-star-def}
  \bm F_*(x) \defeq \frac{2}{\pi} \prod_{i = 1}^d |\sin( \pi x_i )|
  \,.
\end{equation}
Note~$\bm F_*$ is $\Z^d$ periodic, vanishes on~$\cup_{k \in \Z^d} \partial Q_k$.
Moreover, for any cube $Q_k$, $\bm F_*$ restricted to $\bar Q_k$ is the principal Dirichlet eigenfunction of the Laplace operator, normalized to have integral $1$ over $Q_k$.

\begin{lemma}\label{l:FkLower}
  For any $k \in \Z^d$, $s \in \mathcal T$, $\bm F_{k,s}$ is supported on~$\cyl_{k, s}$, strictly positive on~$\mathring \cyl_{k, s}$,  has integral~$1$.
  Moreover, there exists~$a > 0$ such that for all~$k \in \Z^d$, $s \neq \zerotuple$, we have the pointwise inequality
  \begin{equation}\label{e:TStarLower}
    T_{*,\epsilon} \bm F_{k,s}
      \geq
	e^{ -a (\lambda_s \epsilon)^2 }
	\bm F_{\sigma(k,s)}
    \quad\text{where}\quad
    \lambda_{s} = \frac{1}{\abs{\cyl_{k,s}}^{1/d}}
    \,.
  \end{equation}
  Here $T_{*,\epsilon}$ is the evolution operator defined by
  \begin{equation}\label{e:t-star-lower}
    T_{*, \epsilon} f(y) = \int_{\R^d} f(x) p^{X^\epsilon}_1(x, y) \, dx
    \,,
  \end{equation}
  where $p^{X^\epsilon}_1$ is the one step transition density of~$X^\epsilon$.
\end{lemma}
\begin{remark}
  The explicit form of~$\bm F_{k,s}$ in~\eqref{e:FkDef} only works when the noise~$\xi$ is Gaussian.
  When the noise~$\xi_n$ is not a Gaussian, the construction of a suitable family of functions~$\bm F_{k,s}$ surprisingly difficult.
  Following Section 3.5 in~\cite{IyerLuEA24} one can find~$a, \gamma > 0$ and a family of functions~$\bm F_{k, s}$ such that
  \begin{equation}
    T_{*,\epsilon} \bm F_{k,s}
      \geq
	(1 - a (\lambda_{s} \epsilon)^\gamma )
	\bm F_{\sigma(k,s)}
    \,,
  \end{equation}
  which is good enough for our purposes.
\end{remark}

\begin{proof}[Proof of Lemma~\ref{l:FkLower}]
  The proof is identical to that of Lemma 3.1 in~\cite{IyerLuEA24}.
  The only difference is that here the ambient space is $\R^d$ rather than $\T^d$.
  The only change required to the proof of Lemma 3.1 in~\cite{IyerLuEA24} is to additionally keep track of the $\Z^d$-index of cylinder sets (i.e. $(k,s)$ rather than $s$), so that $\varphi$ maps $\cyl_{k, s}$ to $\cyl_{\sigma(k,s)}$.
\end{proof}

\begin{lemma}\label{l:TstarLower}
  Let~$a$ be as in Lemma~\ref{l:FkLower}.
There exist constants $\beta \in (0,1]$ such that for all $\epsilon > 0$ and all $s \in \mathcal T$ such that $\ell_s \geq \epsilon$, we have
  \begin{equation}\label{e:TnStarLower}
    T_{*,\epsilon}^{n} \bm F_{k,s}  \geq \beta \bm F_{\sigma^n(k,s)}, \quad \forall \; n \leq |s|
  \end{equation}
In particular, 
  \begin{equation}\label{e:TStarLower1}
    T_{*,\epsilon}^{|s|} \bm F_{k,s}  \geq \beta \one_{Q_{J(k,s)}}(x) \bm F_*(x), 
  \end{equation}
where $J(k,s)$ is defined in \eqref{Jksdef}.
\end{lemma}

\begin{proof}[Proof of Lemma~\ref{l:TstarLower}]
  Iterating~\eqref{e:TStarLower} shows that for any~$n \leq \abs{s}$ we have
  \begin{equation}\label{e:TStarEpN}
    T_{*, \epsilon}^{n} \bm F_s
      \geq
	\exp\paren[\Big]{
	  -a  \epsilon^2
	  \sum_{j=1}^{n} \lambda_{\sigma^j (k, s)}^2  }
	\bm F_{\sigma^n (k, s)}
	\,.
  \end{equation}
  The sum in the exponential is easily bounded.
  Indeed, if  $s = (s_0, \dots, s_{n'-1}) \in \mathcal T $, then 
  \begin{equation*}
    \lambda_{\sigma (k, s)}
      = \frac{1}{\pi(\cyl_{\sigma (k, s)})^{1/d}}
      = \frac{\abs{E_{s_0}}^{1/d}}{\pi(\cyl_{k, s})^{1/d}}
      \leq \abs{E_M}^{1/d} \lambda_{k, s}\,.
  \end{equation*}
  Hence for every $n \leq \abs{s}$ we have
  \begin{equation}
    a  \epsilon^2 \sum_{j=1}^{n} \lambda_{\sigma^j (k, s)}^2
      \leq \frac{(\epsilon \lambda_{\sigma (k, s)})^2}{1 - \abs{E_M}^{2/d} }
      \leq  \frac{1}{1 - \abs{E_M}^{2/d} }
      \,,
  \end{equation}
  which is finite and independent of $n$ and $\epsilon$.
  Using this in~\eqref{e:TStarEpN} shows that for all $n \leq \abs{s}$,~\eqref{e:TnStarLower} holds with
  \begin{equation}
    \beta = \exp \paren[\Big]{
      \frac{-1}{1 - \abs{E_M}^{2/d} }
    }
    \,.
    \qedhere
  \end{equation}
\end{proof}

\subsection{Defragmenting the density}

Since~$\varphi$ may fragment each cube~$Q_n$, we now use the integer shifts~$I$ to defragment the density with small probability.
This will be used to minorize the distribution of~$X^\epsilon(x)$ after time~$2 + \theta^\epsilon(x)$, and leads to the proof of Lemma~\ref{l:p2z-doeblin}.

\begin{lemma}\label{l:Zprime-lower}
  Suppose~$Z'$ is a random variable and~$c > 0$ are such that for every Borel set~$A \subseteq \R^d$ we have
  \begin{equation}\label{e:Zprime-lower}
    \P( Z' \in A ) \geq c \int_{A \cap Q_0} \bm F_*(x) \, dx
    \,.
  \end{equation}
  Let~$I$ be an independent random variable with distribution~\eqref{e:I1-def}, and~$\xi$ be an independent standard normal.
  There exists a constant~$c' > 0$ such that for all~$\epsilon > 0$ we have
  \begin{equation}
    \P( \varphi(Z') + \epsilon \xi - I \in A ) \geq c' \abs{A \cap Q_0}
  \end{equation}
for all Borel set~$A \subseteq \R^d$.
\end{lemma}
\begin{proof}
For~$i \in \set{1, \dots M}$, define~$\hat \varphi_i \colon \bar E_i \to \bar Q_0$ by
\begin{equation}
  \hat \varphi_i(z)
    = \varphi(z) - \sigma_i
    = \varphi(z) - \floor{\varphi(z)}
  \,.
\end{equation}
Observe that $\hat \varphi_i$ is invertible on $Q_0$. For any Borel set~$A$ we note
\begin{align}
  \P( \varphi(Z') - I \in A )
    & \geq \sum_{i=1}^M 
      \P( Z' \in \mathring E_i,~
	I = \sigma_0(i),~
	\varphi(Z') - I \in A )
  \\
    & \geq p_{\min} \sum_{i=1}^M 
      \P( Z' \in \mathring E_i,~
      \hat \varphi(Z') \in A )
  \\
    & \geq c p_{\min}
	\sum_{i=1}^M \int_{\mathring E_i}
	  \bm F_*(z) \one_A \circ \hat \varphi(z) \, dz
  \\
    & = c p_{\min}
	\int_{Q_0}
	  g(z)
	  \one_A(z)
	  \, dz
	   \,,
\end{align}
where
\begin{equation}
  g(z) \defeq \sum_{i=1}^M
    \bm F_*\circ \hat \varphi_i^{-1}(z)
    \abs{E_i} \, dz
  \,.
\end{equation}

By Assumption~\ref{a:phi} item (\ref{a:int-cube}), and the fact that~$\bm F_* > 0$ in~$\mathring Q_0$ we know that~$g > 0$ on~$\bar Q_0$.
Since~$g$ is continuous on~$\bar Q_0$, this implies
\begin{equation}
  \min_{Q_0} g > 0\,.
\end{equation}
This in turn implies that
\begin{equation}
  \P( \varphi(Z') - I \in A )
    \geq c \abs{A \cap Q_0}\,.
\end{equation}

Let~$f_\epsilon$ be the density of the random variable~$\varphi(Z') - I + \epsilon \xi$, and~$G_{\epsilon^2}$ be the density of a Gaussian with variance~$\epsilon^2$.
The above implies
\begin{equation}
  f_0 \geq c \one_{Q_0}
  \,.
\end{equation}
Thus
\begin{equation}
  f_\epsilon = f_0 * G_{\epsilon} \geq \frac{c}{2^d} \one_{Q_0}
  \,,
\end{equation}
which implies~\eqref{e:Zprime-lower} as desired.

\end{proof}

\begin{lemma}\label{l:p1z-doeblin}
  The one step transition density of the process~$Z^\epsilon$ is bounded below by
  \begin{equation}\label{e:pZ1-lower}
    p_1^{Z^\epsilon}(x, y)
      \geq c \one_{ Q(X^0_{1 + \theta^\epsilon}(x)) } (y)\,.
  \end{equation}
\end{lemma}
\begin{proof}
  The first step is to show that the distribution of~$X^\epsilon_1$ is bounded below by one of the functions~$\bm F_{k, s}$ with~$\ell_s = O(\epsilon)$.
  For this, we will now partition~$\R^d$ into cylinder sets with side length~$O(\epsilon)$ as follows.
  Define
  \begin{equation}\label{e:Sepsilon}
    \mathcal S_\epsilon = \set[\Big]{
      (k, s) \in \Z^d \times \mathcal T \st
	\ell_{(k, s')} > \epsilon
	\,, \quad
	\ell_{k,s} \leq \epsilon
    }
    \,,
  \end{equation}
  where~$s' \in \mathcal S$ is obtained by removing the last element of~$s$.
  That is, if $s = (s_0, \dots, s_{m-1})$ then we define~$s' = (s_0, \dots, s_{m-2})$).
  Observe that~$\set{ \mathcal C_{k,s} \st (k, s) \in \mathcal S_\epsilon}$ is a partition of~$\R^d$ (see Figure~\ref{f:Nepsilon} for an illustration).

  Fix~$x \in \R^d$, and let~$N = \theta^\epsilon(x)$, and let~$f^\epsilon_n$ be the density of the random variable~$X^\epsilon_n(x)$.
  Let~$(k, s) \in \mathcal S_\epsilon$ be the unique element such that~$\varphi(x) \in \mathcal C_{k, s}$.
  Since
  \begin{equation}
    \epsilon \abs{E_1}^{1/d} < \ell_{k, s} \leq \epsilon 
    \,,
  \end{equation}
  and~$X^\epsilon_1(x)$ is a Gaussian with mean~$\varphi(x)$ and covariance matrix~$\epsilon^2 I$, there is a constant~$c$, independent of~$\epsilon, k, s$ and~$x$, such that
  \begin{equation}
    f^\epsilon_1 \geq c \bm F_{k ,s}
    \,.
  \end{equation}

  Now we note that for any~$j \in \Z^d$, $t \in \mathcal T$ and~$y \in \mathcal C_{j, t}$ we have
  \begin{equation}
    \abs{\mathcal C_{j, t}}
      = \prod_{n=0}^{\abs{t}-1}
	  \frac{ \abs{\mathcal C_{\sigma^n(j,t)}} }{ \abs{\mathcal C_{\sigma^{n+1}(j,t)} } }
      = \prod_{n=0}^{\abs{t}-1}
	  \frac{ 1 }{\abs{\det D\varphi( \varphi^n(y) ) }}
  \end{equation}
  Hence for the~$(k, s)$ chosen above we know
  \begin{equation}
    \prod_{n=1}^{\abs{s}} \frac{ 1 }{\abs{\det D\varphi( \varphi^n(y) ) }} \leq \epsilon^d
    \,,
    \quad\text{and}\quad
    \prod_{n=1}^{\abs{s}-1} \frac{ 1 }{\abs{\det D\varphi( \varphi^n(y) ) }} > \epsilon^d
    \,,
  \end{equation}
  The definition of~$\theta^\epsilon$ in~\eqref{e:thetaDef} now implies
  \begin{equation}\label{e:NThetaEps}
    N = \theta^\epsilon(x) = \abs{s}
    \,.
  \end{equation}
  Using Lemma~\ref{l:TstarLower} we obtain
  \begin{equation}
    f_{N+1} \geq c \beta \one_{Q(J(k,s))} \bm F_*
    \,,
  \end{equation}
  following which Lemma~\ref{l:Zprime-lower} yields
  \begin{equation}
    f_{N+2} \geq c \beta \one_{Q(J(k,s))}
    \,.
  \end{equation}
  Since
  \begin{equation}
    J(k,s )
      = \floor{ \varphi^N ( \varphi(x) ) }
      = \floor{ X^0_{1+\theta^\epsilon}(x) }
    \,,
  \end{equation}
  this implies~\eqref{e:pZ1-lower} as claimed.
\end{proof}

Lemma~\ref{l:p2z-doeblin} follows immediately from Lemma~\ref{l:p1z-doeblin}, and we conclude this section by presenting the proof.
\begin{proof}[Proof of Lemma~\ref{l:p2z-doeblin}]
  Using Lemma~\ref{l:p1z-doeblin} and the Markov property we see
  \begin{align}
    p_2^{Z^\epsilon}(x, y)
      &= \int_{\R^d} p_1^{Z^\epsilon}(x, x') p_1^{Z^\epsilon}(x', y) \, dx'
      \geq c \int_{Q(X^0_{1 + \theta^\epsilon}(x))} p_1^{Z^\epsilon}(x', y) \, dx'
    \\
      & \geq c^2 \int_{Q(X^0_{1 + \theta^\epsilon}(x))}
	\one_{Q(X^0_{1 + \theta^\epsilon}(x')} (y) \, dx'
      = c^2 \check w^\epsilon(x, )
  \end{align}
  where~$\check w^\epsilon$ is defined in~\eqref{e:wEpCheckDef}.
\end{proof}

\subsection{Bounding the effective diffusivity (Lemma \ref{l:DwCheck})}

Lemma~\ref{l:DwCheck} follows immediately from the Bernoulli structure of~$\varphi$, and we present the proof here.

\begin{proof}[Proof of Lemma~\ref{l:DwCheck}]
  Let~$x \in Q_0$ be arbitrary, $U \sim \unif(Q_0)$ and define
  \begin{equation}
    K_n = \floor{X^0_n(U)}
    \,.
  \end{equation}
  The Bernoulli structure of~$\varphi$ ensures that~$K_n$ is a random walk on~$\Z^d$ with i.i.d.\ increments.
  To see this, define the filtration~$\set{\mathcal F_n}$ by
  \begin{equation}
    \mathcal F_n = \set{ U \in \mathcal C_{0, s} \st \abs{s} \leq n }
    \,.
  \end{equation}
  For any event~$A \in \mathcal F_n$, item~\ref{a:int-cube} in Assumption~\ref{a:phi} will ensure that the density of~$\one_{A} X^0_n(U)$ is constant on each of the cubes~$\set{Q_m}_{m \in \Z^d}$.
  Since
  \begin{align}
    \P( K_{n+1} - K_n = i \given X^0_n(U) \in Q_k )
      = \P( \varphi(V) - k = i \given V \sim \unif(Q_k) )
      = \abs{E_i}
      \,,
  \end{align}
  the process~$K$ must have i.i.d.\ increments.

  Now let~$\mathcal S_\epsilon$ be as in~\eqref{e:Sepsilon}, and using~\eqref{e:thetaDef} (as we did in~\eqref{e:NThetaEps}) we obtain
  \begin{equation}
    \set{\theta^\epsilon(U) = n}
      = \bigcup \set{ U \in \mathcal C_{0, s} \st \mathcal C_{0, s} \in \mathcal S_\epsilon\,,~ \abs{s} = n}
      \in \mathcal F_n
      \,.
  \end{equation}
  This implies~$\tau = 1 + \theta^\epsilon(U)$ is a stopping time for the filtration~$\mathcal F_n$.

  Thus, for any~$u, v \in \R^d$ we have
  \begin{align}
    u \cdot \mathcal D_{\check w^\epsilon(x)} v 
      &= \cov( u \cdot \floor{K_\tau}, v \cdot \floor{K_\tau} )
      = \E \tau \cov( u \cdot \floor{K_1}, v \cdot \floor{K_1} )
    \\
      &= (1 + {\bar \theta}^\epsilon )
	u \cdot \mathcal D_{w^0} v\,,
  \end{align}
  proving~\eqref{e:DwCheck1}.
\end{proof}

\bibliographystyle{halpha-abbrv}
\bibliography{preprints,gautam-refs1,gautam-refs2}
\end{document}